\documentclass[11pt]{article}  

\usepackage[utf8]{inputenc}
\usepackage{amsmath}
\usepackage{mathtools}
\usepackage{cite}
\usepackage{amsfonts}
\usepackage{amssymb}
\usepackage{optidef}
\usepackage[noabbrev]{cleveref}
\usepackage{comment}
\usepackage{algorithm}
\usepackage{algpseudocode}
\usepackage[nodisplayskipstretch]{setspace}
\usepackage{url}

\usepackage{graphics} 
\usepackage{epsfig} 
\usepackage{times} 
\usepackage{amsmath} 
\usepackage{amssymb}  
\usepackage{amsfonts}
\usepackage{amsthm}
\usepackage[margin=1in]{geometry}
\usepackage[numbers,sort,compress]{natbib}

\newtheorem{theorem}{Theorem}
\newtheorem{corollary}{Corollary}[theorem]
\newtheorem{definition}{Definition}
\newtheorem{assumption}{Assumption}
\newtheorem{remark}{Remark}
\newtheorem{example}{Example}

\usepackage{xcolor}

\title{\LARGE \bf
Distributed and Localized Closed Loop Model Predictive Control via System Level Synthesis
}

\author{Carmen Amo Alonso\thanks{C. Amo Alonso is a graduate student in the Computing and Mathematical Sciences Department at California Institute of Technology, Pasadena, CA.
        {\tt\small camoalon@caltech.edu}} \and Nikolai Matni
        \thanks{N. Matni is an Assistant Professor with the Department of Electrical and Systems Engineering at the University of Pennsylvania, Philadelphia, PA.    {\tt\small nmatni@seas.upenn.edu}}}%
\date{September 22, 2019, Revised: \today}

\begin{document}

\maketitle
\thispagestyle{empty}
\pagestyle{empty}


\begin{abstract}
We present the Distributed and Localized Model Predictive Control (DLMPC) algorithm for large-scale structured linear systems, a distributed \emph{closed loop} model predictive control scheme wherein only local state and model information needs to be exchanged between subsystems for the computation and implementation of control actions. We use the System Level Synthesis (SLS) framework to reformulate the centralized MPC problem as an optimization problem over closed loop system responses, and show that this allows us to naturally impose localized communication constraints between sub-controllers.
We show that the structure of the resulting optimization problem can be exploited to develop an Alternating Direction Method of Multipliers (ADMM) based algorithm that allows for distributed and localized computation of distributed closed loop control policies.  
We conclude with numerical simulations to demonstrate the usefulness of our method, in which we show that the computational complexity of the subproblems solved by each subsystem in DLMPC is independent of the size of the global system.  To the best of our knowledge, DLMPC is the first MPC algorithm that allows for the scalable distributed computation of distributed closed loop control policies.

\end{abstract}

\section{INTRODUCTION}

Model Predictive Control (MPC) has seen widespread success across many applications. However, the recent need to control increasingly large-scale, distributed, and networked systems has limited its applicability. Large-scale distributed systems are often impossible to control with a centralized controller, and moreover, even when such a centralized controller can be implemented, the high computational demand of MPC renders it impractical. Thus, efforts have been made to develop \emph{distributed} MPC (DMPC) algorithms, wherein sub-controllers solve a local optimization problem, and potentially coordinate with other sub-controllers.  

The majority of DMPC research has focused on open-loop approaches, which, following the discussion in \cite{conte_distributed_2016}, can be broadly categorized into non-cooperative  and cooperative settings.  
In the non-cooperative setting (see for example \cite{farina_distributed_2012}), sub-controllers do not coordinate their actions with each other, and treat other subsystems as disturbances: while computationally efficient, such approaches are known to be conservative, and can even lead to infeasible problems, when there is strong dynamic coupling between subsystems.  In the cooperative setting, sub-controllers exchange state and control action information in order to coordinate their behavior so as to optimize a global objective, typically through distributed optimization: see for example \cite{venkat_distributed_2008,zheng_networked_2013,giselsson_accelerated_2013,conte_distributed_2016,jalal_limited-communication_2017,wang_distributed_2015, venkat2005stability}.  
In order to make these nominal open loop approaches robust to additive disturbances, two broad approaches have been taken to generate closed loop policies.  

The first extends centralized robust MPC techniques that rely on a pre-computed stabilizing controller, such as constraint tightening and tube MPC, to the distributed setting.  While conceptually appealing and computationally efficient, they often rely on strong assumptions, such as the existence of a \emph{static structured} stabilizing controller, as in \cite{conte2013robust}, which can be NP-hard to compute \cite{blondel1997np}, or on dynamically decoupled subsystems, as in \cite{richards2007robust}.  The alternative approach, and that which is adopted in this paper, is to compute a \emph{dynamic} structured feedback policy using a suitable parameterization.  The first paper to propose such a strategy was the seminal paper by Goulart et al. \cite{goulart_optimization_2006}, where it was shown that using a disturbance based parameterization of the control policy allowed for distributed (structured) control policies to be synthesized using convex optimization.  A similar approach exploiting Quadratic Invariance \cite{rotkowitz_characterization_2006} and the Youla parameterization was also recently developed in \cite{furieri_robust_2017}.  While these methods allow for convex optimization to be used for the synthesis of distributed closed loop control policies, the resulting optimization problems lack the structure needed for them to be amenable to distributed optimization techniques, limiting their applicability to smaller scale systems.

Thus, the desiderata for a closed loop DMPC algorithm are that it allow for (i) structured feedback policies to be computed via convex optimization, and (ii) for this computation to be solvable at scale via distributed optimization techniques: to the best of our knowledge, no method satisfying both requirement exists. In this paper we address this gap and present the Distributed Localized MPC (DLMPC) algorithm for linear time-invariant systems, which allows for the distributed computation of structured feedback policies.  

We leverage the System Level Synthesis \cite{wang2019system,wang_separable_2018,anderson_system_2019} (SLS) framework to define a novel parameterization of distributed closed loop MPC policies such that the resulting synthesis problem is \emph{both convex and structured}, allowing for the natural use of distributed optimization techniques.  We show that by 
exploiting the sparsity of the underlying distributed system and resulting closed loop system, as well as the separability properties \cite{wang_separable_2018} of often used objective functions and constraints (e.g., quadratic costs subject to polytopic constraints), we are able to distribute the computation via the Alternating Direction Method of Multipliers (ADMM), thus allowing for the online computation of closed loop MPC policies to be done in a scalable localized manner. Hence, in the resulting implementation each sub-controller solves a low-dimensional optimization problem defined over its local neighborhood, requiring only local communication of state and model information.  We show that so long as certain localizability properties are satisfied, no approximations are needed, and that under standard regularity assumptions, the algorithm converges to the globally optimal solution. Furthermore, we show that convex constraints and cost functions that couple neighboring subsystems can be dealt with via a consensus-like algorithm.  
Through numerical experiments, we further confirm that the complexity of the subproblems solved at each subsystem scales as $O(1)$ relative to the full size of the system.

\subsubsection*{\textbf{Notation}}

Bracketed indices denote the time of the true system, i.e., $x(t)$ denotes the system state at time $t$, subscripts denote prediction time indices within an MPC loop, i.e., $x_t$ denotes the $t$th predicted state, and superscripts denote iteration steps of an optimization algorithm, i.e., $x_t^{k}$ is the value of the $t^{\text{th}}$ predicted state at iterate $k$. To denote subsystem variables, we use square bracket notation, i.e., $[x]_{i}$ denotes the components of $x$ corresponding to subsystem $i$.  Calligraphic letters such as $\mathcal{S}$ denote sets, and lowercase script letters such as $\mathfrak{c}$ denotes a subset of $\mathbb{Z}^{+}$, i.e., $\mathfrak{c}=\left\{1,...,n\right\}\subset\mathbb{Z}^{+}$.  Boldface lower and upper case letters such as $\mathbf{x}$ and $\mathbf{K}$ denote finite horizon signals and lower block triangular (causal) operators, respectively:
\begin{equation*} 
\mathbf{x}=\left[\begin{array}{c} x_{0}\\x_{1}\\\vdots\\x_{T}\end{array}\right],\
    \mathbf{K}=\left[\begin{array}{cccc}K^{0,0} & & & \\ K^{1,1} & K^{1,0} & & \\ \vdots & \ddots & \ddots & \\ K^{T,T} & \dots & K^{T,1} & K^{T,0} \end{array}\right],
\end{equation*}
where each $K^{i,j}$ is a matrix of compatible dimension. 
  $\mathbf{K}(\mathfrak{r},\mathfrak{c})$ denotes the submatrix of $\mathbf{K}$ composed of the rows specified by $\mathfrak{r}$ and the columns specified by $\mathfrak{c}$. 

\section{PROBLEM FORMULATION}

Consider a discrete-time linear time invariant (LTI) system
\begin{equation} \label{eq: LTV system}
x(t+1) = Ax(t)+Bu(t)+w(t),
\end{equation}
where $x(t)\in\mathbb{R}^{n}$ is the state, $u(t)\in\mathbb{R}^{p}$ is the control input, and $w(t)\in\mathbb{R}^{n}$ is an exogenous disturbance. The system is composed of $N$ interconnected subsystems, as defined by an interconnection topology -- correspondingly, the state, control, and disturbance inputs can be suitably partitioned as $[x]_i$, $[u]_i$, and $[w]_i$, inducing a compatible block structure $[A]_{ij}$, $[B]_{ij}$ in the dynamics matrices $(A,B)$.   We model the interconnection topology of the system as a time-invariant
\footnote{Although we restrict both the dynamics and interconnection topology to be time-invariant, we believe that an extension to time-varying dynamics and topologies will be straightforward as long as the corresponding communication topology varies consistently with the physical topology. We leave this extension for future work.}
unweighted directed graph $\mathcal{G}_{(A,B)}(E,V)$, where each subsystem $i$ is identified with a vertex $v_{i}\in V$ and an edge $e_{ij}\in E$ exists whenever $[A]_{ij}\neq 0$ or $[B]_{ij}\neq 0$.

\begin{example}{\label{ex: chain graph}}
Consider the linear time-invariant system structured as a chain topology as shown in Figure \ref{fig:chain}. Each subsystem $i$ is subject to the dynamics
\begin{equation*}
[x(t+1)]_{i} = \sum_{j\in\left\{i,i\pm 1\right\}}[A]_{ij}[x(t)]_{j}+[B]_{ii}[u(t)]_{i}+[w(t)]_{i}.
\end{equation*}

As $B$ is a diagonal matrix, coupling between subsystems is defined by the $A$ matrix -- thus, the adjacency matrix of the corresponding graph $\mathcal{G}$ coincides with the support of $A$.
\begin{figure}[htp]
\centering
\includegraphics[width=11cm, page=1]{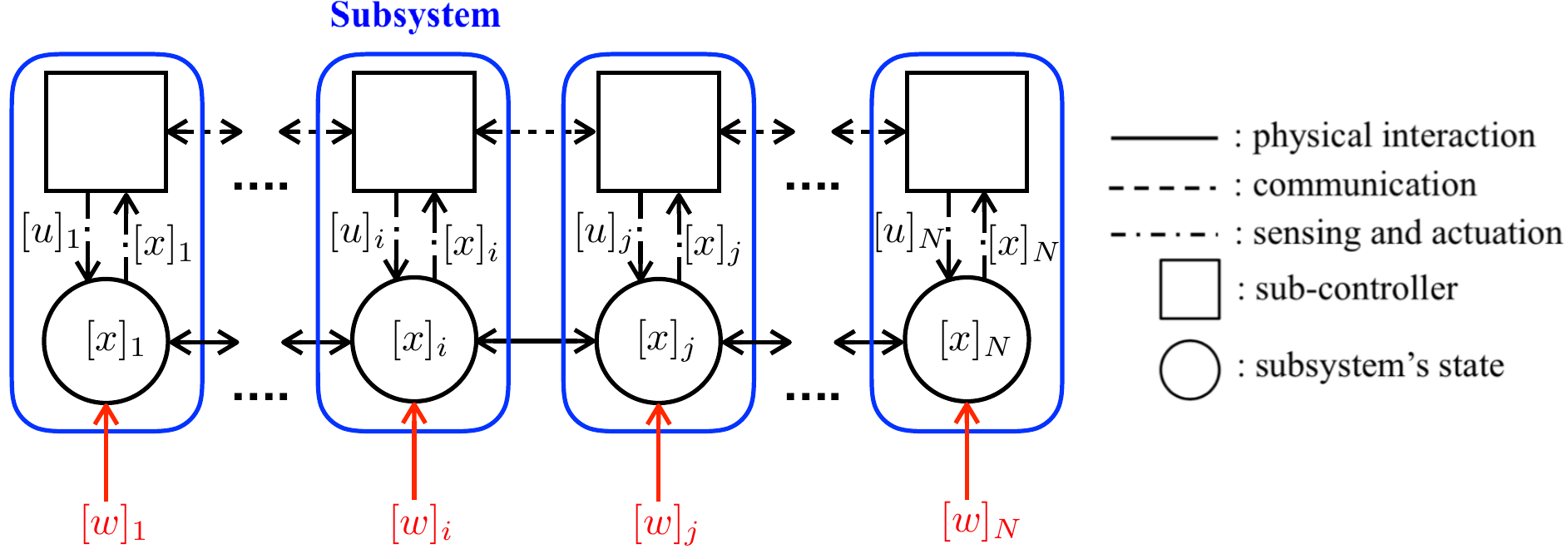}
\caption{Schematic representation of a system with a chain topology.}
\label{fig:chain}
\end{figure}

\end{example} 

As is standard, a model predictive controller is implemented by solving a series of finite horizon optimal control problems, with the problem solved at time $\tau$ with initial condition $x_0=x(\tau)$ over a prediction horizon $T$ given by: 
\begin{equation} \label{eq: MPC}
\begin{aligned}
& \underset{{x}_{t},u_{t}, \gamma_t}{\text{min.}} &  &\sum_{t=0}^{T-1}f_{t}(x_{t},u_{t})+f_{T}(x_{T})\\
& \ \text{s.t.} &  &\begin{aligned} 
    & x_{0} = x(\tau),\, \ x_{t+1} = Ax_{t}+Bu_{t},\, \ t=0,...,T-1,\\
    & x_{T}\in\mathcal{X}_{T},\, x_{t}\in\mathcal{X}_{t},\, u_{t}\in\mathcal{U}_{t} \, \ t=0,...,T-1, \\
    &u_{t} = \gamma_t(x_{0:t},u_{0:t-1}),
\end{aligned}
\end{aligned}
\end{equation}
where the $f_t(\cdot,\cdot)$ and $f_T(\cdot)$ are convex cost functions, $\mathcal{X}_t$ and $\mathcal{U}_t$ are convex sets containing the origin, and $\gamma_t(\cdot)$ are measurable functions of their arguments.

Our goal is to define a MPC algorithm that respects local communication constraints between sub-controllers, and has only local-scale computational complexity when computing and subsequently implementing distributed closed loop control policies.  In what follows, we formally define appropriate notions of locality in terms of the interconnection topology graph $\mathcal G_{(A,B)}$ of the underlying physical system, and relate these to the corresponding constraints that they impose on the MPC problem \eqref{eq: MPC}.

We assume that the information exchange topology between sub-controllers matches that of the underlying system, i.e., that it is given by $\mathcal{G}_{(A,B)}$, and we further impose that information exchange be \emph{localized} to a subset of neighboring sub-controllers.  In particular, we use the notion of a $d$-local information exchange constraint \cite{wang2014localized,wang2016localized} to be one that restricts sub-controllers to exchange their state and control actions with neighbors at most $d$-hops away, as measured by the communication topology $\mathcal{G}_{(A,B)}$. This notion is captured by the $d$-outgoing and $d$-incoming sets of subsystem.
\begin{definition}{\label{def: in_out set}}
For a graph $\mathcal{G}(V,E)$, the \textit{d-outgoing set} of subsystem $i$ is $\textbf{out}_{i}(d) := \left\{v_{j}\ |\  \textbf{dist}(v_{i} \rightarrow v_{j} ) \leq d\in\mathbb{N} \right\}$. The \textit{d-incoming set} of subsystem $i$ is $\textbf{in}_{i}(d) := \left\{v_{j}\ |\ \textbf{dist}(v_{j} \rightarrow v_{i} ) \leq d\in\mathbb{N} \right\}$.  Note that $v_i \in \textbf{out}_{i}(d)\cap \textbf{in}_{i}(d)$ for all $d\geq 0$. 
\end{definition}

Hence, we can enforce a $d$-local information exchange constraint on the distributed MPC problem \eqref{eq: MPC} by imposing the constraint that each sub-controllers policy respects 
\begin{equation}
    [u_t]_i = \gamma_{i,t}\left([x_{0:t}]_{j\in \textbf{in}_i(d)},[u_{0:t-1}]_{j\in \textbf{in}_i(d)},
    [A]_{j,k \in \textbf{in}_i(d)},[B]_{j,k\in  \textbf{in}_i(d)}\}\right),
    \label{eq:info-constraints}
\end{equation}
for all $t=0,\dots,T$ and $i=1,\dots,N$, where $\gamma_{i,t}$ is a measurable function of its arguments.  In words, this says that the closed loop control policy at sub-controller $i$ can be computed using only states, control actions, and system models collected from $d$-hop incoming neighbors of subsystem $i$ in the communication topology $\mathcal{G}_{(A,B)}.$

\begin{example}\label{ex: in_out} Consider a system \eqref{eq: LTV system} composed of $N=6$ scalar subsystems, with $B = I_6$ and $A$ matrix with support represented in Figure \ref{fig: in_out}(a).
This induces the interconnection topology graph $\mathcal G_{(A,B)}$ illustrated in Figure \ref{fig: in_out}(b). The $d$-incoming and $d$-outgoing sets can be directly read off from the interaction topology.  For example, for $d=1$, the $1$-hop incoming neighbors for subsystem $5$ are subsystems $3$ and $4$, hence $\textbf{in}_{5}(1)=\{3,4,5\}$; similarly, $\textbf{out}_{5}(1)=\{4,5,6\}$.

\begin{figure}[h]
    \centering
    \includegraphics[width=12cm, page=2]{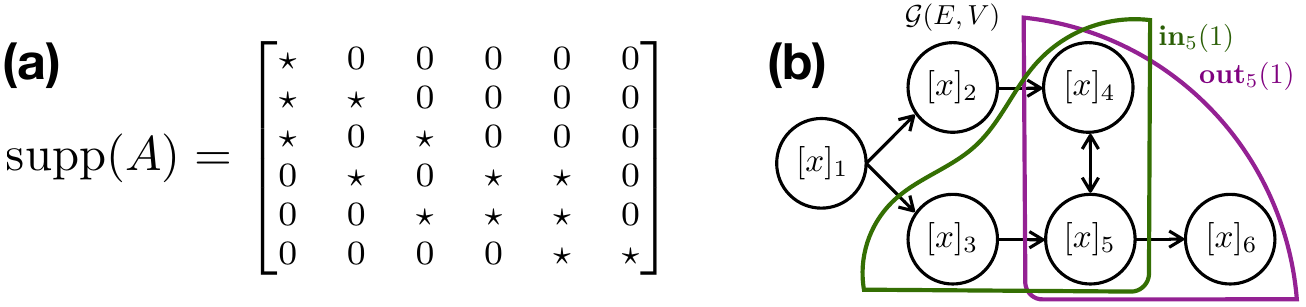}
    \caption{{(a) Support of matrix $A$. (b) Example of $1$-incoming and $1$-outgoing sets for subsystem $5$.}}
    \label{fig: in_out}
\end{figure}
\end{example}

Given such an interconnection topology, it would be desirable to be able to specify that both the synthesis and implementation of a control action at each subsystem be localized, i.e., depend only on state, control action, and plant model information from $d$-hop neighbors, where the size of the local neighborhood $d$ is a \emph{design parameter}.  We will show that, 
 structural compatibility assumptions between the cost function, state and input constraints, and information exchange constraints, DLMPC allows for precisely this by imposing appropriate $d$-local structural constraints on the \emph{closed loop system responses} of the system.  We will make clear that the localized region parameter $d$ allows for a principled trade-off between the amount of coordination allowed between sub-controllers and the computational complexity of the distributed MPC controller.  To do so, we leverage the SLS framework to reformulate the MPC problem \eqref{eq: MPC}. 

\section{SYSTEM LEVEL SYNTHESIS BASED DLMPC\label{section2}}

We first introduce relevant tools from the SLS framework\cite{wang2019system,wang_separable_2018,anderson_system_2019}, and show how SLS naturally allows for locality constraints \cite{wang_separable_2018,wang2014localized,wang2016localized} to be imposed on the system responses and corresponding controller implementation.

\subsection{Time Domain System Level Synthesis}

The following is adapted from $\S2$ of \cite{anderson_system_2019}.  Consider the dynamics of system (\ref{eq: LTV system}) evolving over a finite horizon $t = 0,...T$, and let $u_t$ be a causal linear time-varying state-feedback controller, i.e., $u_t=K_t(x_0,x_1,...,x_t)$ where $K_t$ is some linear map to be designed.\footnote{Our assumption of a linear policy is without loss of generality, as an affine control policy $u_t = K_t(x_{0:t}) + v_t$ can always be written as a linear policy acting on the homogenized state $\tilde{x} = [x;1]$.} Let $Z$ be the block-downshift matrix, i.e., a matrix with identity matrices along its first block sub-diagonal and zeros elsewhere, and define $\hat A:=\mathrm{blkdiag}(A,A,...,A,0)$ and $\hat B:=\mathrm{blkdiag}(B,B,...,B,0)$. This allows us to write the behavior of system (\ref{eq: LTV system}) over the horizon $t=0,...,T$ as 
\begin{equation} 
\mathbf{x} = Z(\hat A+ \hat B\mathbf{K})\mathbf{x+\mathbf{w}},
\label{eq: closed loop}
\end{equation}
where $\mathbf{x}$, $\mathbf{u}$ and $\mathbf{w}$ are the finite horizon signals corresponding to state, control input, and disturbance respectively. In particular, the initial condition $x_{0}$ is embedded as the first element of the disturbance, i.e., $\mathbf{w} = [x_{0}^\mathsf{T}\ w_{0}^\mathsf{T}\ \dots \ w_{T-1}^\mathsf{T}]^\mathsf{T}$.

The closed loop behavior of system (\ref{eq: LTV system}) under the feedback law $\mathbf{K}$ can be entirely characterized by {the system responses $\mathbf \Phi_x$ and $\mathbf \Phi_u$}
\begin{equation} \label{eq: closed loop A&B}
\begin{split}
\mathbf{x} & = (I-Z(\hat A+\hat B\mathbf{K}))^{-1}\mathbf{w} =: \mathbf\Phi_x \mathbf w\\
\mathbf{u} & = \mathbf{K}(I-Z(\hat A+\hat B\mathbf{K}))^{-1}\mathbf{w} =: \mathbf\Phi_u \mathbf w.
\end{split}
\end{equation}

The approach taken by SLS is to directly parameterize and optimize over the set of achievable closed loop maps (\ref{eq: closed loop A&B}) from the exogenous disturbance $\mathbf{w}$ to the state $\mathbf{x}$ and the control input $\mathbf{u}$, respectively.
\begin{theorem}{\label{thm: SLS}}
For the dynamics (\ref{eq: LTV system}) evolving under the state-feedback policy $\mathbf u = \mathbf K \mathbf x$, for $\mathbf{K}$ a block-lower-triangular matrix, the following are true
\begin{enumerate}
    \item The affine subspace of block lower-triangular $\{\mathbf{\Phi}_{x},\mathbf{\Phi}_{u}\}$
    \begin{equation}\label{eq: constraint}
        \left[I-Z\hat A\ \ -Z\hat B\right]\left[\begin{array}{c}\mathbf{\Phi}_{x}\\\mathbf{\Phi}_{u}\end{array}\right] = I
    \end{equation}
    parameterizes all possible system responses \eqref{eq: closed loop A&B}.
    
    \item For any block-lower-triangular matrices $\left\{\mathbf{\Phi}_{x},\mathbf{\Phi}_{u}\right\}$ satisfying (\ref{eq: constraint}), the controller $\mathbf{K} = \mathbf{\Phi}_{u}\mathbf{\Phi}_{x}^{-1}$ achieves the desired response \eqref{eq: closed loop A&B} from $\mathbf w \mapsto (\mathbf x,\mathbf u)$.
\end{enumerate}
\end{theorem}

\begin{proof}See Theorem 2.1 of \cite{anderson_system_2019}.\end{proof}

Theorem \ref{thm: SLS} allows us to reformulate an optimal control problem over state and input pairs $(\mathbf x, \mathbf u)$ as an equivalent one over system responses $\{\mathbf \Phi_x, \mathbf \Phi_u\}$ -- a detailed description of how to do this for several standard control problems is provided in $\S2$ of \cite{anderson_system_2019}. For the MPC subproblem \eqref{eq: MPC}, as no driving noise is present, we only have to account for the system response to the initial condition $x_0$, i.e., $\mathbf{w}=[x_{0}^{\mathsf{T}},0,...,0]^{\mathsf{T}}$. Hence, by equation \eqref{eq: closed loop A&B}, $\mathbf{x}=\mathbf{\Phi}_{x}[0]x_{0}$ and $\mathbf{u}=\mathbf{\Phi}_{u}[0]x_{0}$, where $\mathbf\Phi_x[0]$ and $\mathbf\Phi_u[0]$ denote the first block column of the block lower triangular response matrices $\mathbf\Phi_x$ and $\mathbf\Phi_u$ -- in the sequel, we will sometimes abuse notation and write $\mathbf\Phi_x[0] = \mathbf\Phi_x$, $\mathbf\Phi_u[0] = \mathbf\Phi_u$, as in the absence of driving noise, only the first block columns of $\mathbf \Phi_x,\mathbf \Phi_u$ need to be computed. We can rewrite the MPC subproblem \eqref{eq: MPC} as
\begin{equation}\label{eq: SLS}
\begin{array}{rl}
\underset{\mathbf{\Phi}_{x},\mathbf{\Phi}_{u}}{\text{min}} & f(\mathbf{\Phi}_{x}x_{0},\mathbf{\Phi}_{u}x_{0})\\
 \text{s.t.} &  Z_{AB}\mathbf{\Phi}=I,\, x_0 = x(t),\, \mathbf{\Phi}_{x}x_{0}\in \mathcal{X}^T, \mathbf{\Phi}_{u}x_{0} \in\mathcal{U}^T,
\end{array}
\end{equation}
where we use $ Z_{AB}\mathbf{\Phi}=I$ to compactly denote constraint \eqref{eq: constraint}, $\mathcal{X}^T := \otimes_{t=0}^{T-1} \mathcal{X}_t \otimes \mathcal{X}_T$, and similarly for $\mathcal{U}^T$, and $f$ is suitably defined such that it is consistent with the objective function of problem \eqref{eq: MPC}.  

To see why optimization problem \eqref{eq: SLS} is equivalent to the original MPC problem \eqref{eq: MPC}, it suffices to notice that for the noise free setting we consider in this paper, for a fixed initial condition $x_0$ any control sequence $\boldsymbol{u}(x_0) := [u_0^\top,\dots,u_{T-1}^\top]^\top$ can be achieved by a suitable choice of feedback matrix $\boldsymbol K(x_0)$ such that $\boldsymbol{u}(x_0) = \boldsymbol K(x_0)[0]x_0$ (that such a matrix always exists follows from a simple dimension counting argument).  As this control action can be achieved by a linear-time-varying controller $\boldsymbol K(x_0)$, Theorem \ref{thm: SLS} states that there exists a corresponding achievable system response pair $\left\{\mathbf{\Phi}_{x},\mathbf{\Phi}_{u}\right\}$ such that $\boldsymbol{u}(x_0) = \boldsymbol \Phi_u[0]x_0$.  This is simply a restatement in the SLS parameterization of the well known fact that LTV controllers are as expressive as nonlinear controllers over a finite horizon, given a fixed initial condition and noise realization (which in this case is set to zero). Thus the SLS reformulation introduces no conservatism relative to open-loop MPC in the nominal (disturbance free) setting. We defer discussion of the closed loop setting to the end of this section, where we show that the disturbance based parametrization \cite{goulart_optimization_2006} is as a special case of ours. 

\textbf{Why use SLS for Distributed MPC:} In the centralized setting, where both the system matrices $(A,B)$ and the system responses $\left\{\mathbf{\Phi}_{x},\mathbf{\Phi}_{u}\right\}$ are dense, the SLS parameterized problem \eqref{eq: SLS} is slightly more computationally costly than the original MPC problem \eqref{eq: MPC}, as there are now $n(n+p)T$ decision variables, as opposed to $(n+p)T$ decision variables.  However, under suitable localized structural assumptions on the objective functions $f_t$ and constraint sets $\mathcal X^T$ and $\mathcal U^T$, that by lifting to this higher dimensional parameterization, the structure of the underlying system - as captured by the interconnection topology $\mathcal G_{(A,B)}$ - can be fully exploited. In particular, the interplay of the high dimensional parametrization and the system structure allows for not only the convex synthesis of a distributed closed loop control policy (as is similarly done in \cite{goulart_optimization_2006,furieri_robust_2017}), but also for the solution of this convex synthesis problem to be computed using distributed optimization.

This latter feature is one of our main contributions, and in particular, we show that the resulting number of optimization variables in the local subproblems solved at each sub-system scales as $O(d^2T)$, and hence is independent of the global system size.  To the best of our knowledge, this is the first distributed closed loop MPC algorithm that enjoys such properties.

\subsection{Locality in System Level Synthesis} 

We begin by commenting on controller implementation.  Given a pair of achievable system responses $\left\{\mathbf{\Phi}_{x},\mathbf{\Phi}_{u}\right\}$ satisfying the affine constraint \eqref{eq: constraint}, the control law achieving the desired system behavior can be implemented as
\begin{equation}{\label{eq: implementation}}
\begin{aligned}
        \mathbf{u}=\mathbf{\Phi}_u\mathbf{\hat{w}},\ \ \  \mathbf{\hat{x}}=(\mathbf{\Phi}_x-I)\mathbf{\hat{w}},\ \ \ 
        \mathbf{\hat{w}}=\mathbf{x}-\mathbf{\hat{x}},
\end{aligned}
\end{equation}
where $\mathbf{\hat{x}}$ can be interpreted as a nominal state trajectory, and $\mathbf{\hat{w}}=Z\mathbf{w}$ is a delayed reconstruction of the disturbance.  Note that $\Phi_x -I$ is strictly lower block triangular (i.e., strictly causal), and thus the implementation defined in \eqref{eq: implementation}  is well posed \cite{anderson_system_2019}. We further note that in the simplified setting of no driving noise, this implementation reduces to $\mathbf u = \mathbf \Phi_u[0] x_0 $. The advantage of this controller implementation, as opposed to  $\mathbf u = \mathbf \Phi_u \mathbf \Phi_x^{-1} \mathbf x$, is that any structure imposed on the maps $\{\mathbf\Phi_u, \mathbf \Phi_x\}$ translates directly to structure on controller implementation \eqref{eq: implementation}, naturally allowing for information exchange constraints to be imposed by imposing suitable sparsity structure on the responses $\left\{\mathbf{\Phi}_{x},\mathbf{\Phi}_{u}\right\}$.

We now show that imposing $d$-local structure on the system responses, coupled with an assumption of compatible $d$-local structure on the objective functions and constraints of the MPC problem \eqref{eq: MPC}, leads to a \emph{structured} SLS MPC optimization problem \eqref{eq: SLS}.  We then develop an ADMM based distributed solution to this problem in Section \ref{sec:admm}.  We emphasize that while the results of \cite{goulart_optimization_2006, furieri_robust_2017} allow for similar structural constraints to be imposed on the controller realization through the use of either disturbance feedback or Youla parameterizations (subject to Quadratic Invariance \cite{rotkowitz_characterization_2006} conditions), the resulting synthesis problems do not enjoy the structure needed for distributed optimization techniques to be effective, thus limiting their usefulness to smaller scale examples where centralized computation of policies is feasible.  We return to argue this point more formally at the end of this section, after introducing the necessary concepts.

We begin by defining the notion of $d$-localized system responses, which follows naturally from the notion of $d$-local information exchange constraints. They consist of system responses with suitable sparsity patterns such that the information exchange needed between subsystems to implement the controller realization \eqref{eq: implementation} is limited to $d$-hop incoming and outgoing neighbors, as defined by the topology $\mathcal G_{(A,B)}$. 

\begin{definition}{\label{def: locality}}
Let $[\mathbf{\Phi}_{x}]_{ij}$ be the submatrix of system response $\mathbf{\Phi}_x$ describing the map from disturbance $[w]_{j}$ to the state $[x]_i$ of subsystem $i$. The map $\mathbf{\Phi}_{x}$ is \textit{d-localized} if and only if for every subsystem $j$, $[\mathbf{\Phi}_{x}]_{ij}=0\ \forall\ i\not\in\textbf{out}_{j}(d)$. The definition for \textit{d-localized} $\mathbf{\Phi}_u$ is analogous but with perturbations to control action $[u]_i$ at subsystem $i$.
\end{definition}

It follows immediately from the controller implementation (\ref{eq: implementation}) that if the system responses are $d$-localized, then so is the controller implementation. In particular, by enforcing $d$-localized structure on $\mathbf{\Phi}_{x}$, only a corresponding local subset $[\hat{\mathbf{w}}]_{j\in\textbf{in}_i(d)}$ of $\hat{\mathbf{w}}$ are necessary for subsystem $i$ to compute its local disturbance estimate $[\hat{\mathbf{w}}]_{i}$, which ultimately means that only local communication is required to reconstruct the relevant disturbances for each subsystem. Similarly, if $d$-localized structure is imposed on $\mathbf{\Phi}_{u}$, then only a local subset $[\hat{\mathbf{w}}]_{j\in\textbf{in}_i(d)}$ of the estimated disturbances $[\hat{\mathbf{w}}]$ is needed for each subsystem to compute its control action $[\mathbf u]_i$. Hence, each subsystem only needs to collect information from its $d$-incoming set to implement the control law defined by \eqref{eq: implementation}, and similarly, only needs to share information with its $d$-outgoing set to allow for other subsystems to implement their respective control laws.
Furthermore, such locality constraints are transparently enforced as additional subspace constraints in the SLS formulation (\ref{eq: SLS}).
\begin{definition}{\label{def: locality constraints}}
A subspace $\mathcal{L}_d$ enforces a $d$\textit{-locality constraint} if $\mathbf{\Phi}_{x},\mathbf{\Phi}_{u}\in\mathcal{L}_d$ implies that $\mathbf{\Phi}_{x}$ is $d$-localized and $\mathbf{\Phi}_{u}$ is $(d+1)$-localized\footnote{Notice that we are imposing $\mathbf{\Phi}_{u}$ to be $(d+1)$-localized because in order to localize the effects of a disturbance within the region of size $d$, the ``boundary'' controllers at distance $d+1$ must take action (for more details the reader is referred to \cite{anderson_system_2019}).}.  A system $(A,B)$ is then $d$-localizable if the intersection of $\mathcal{L}_d$ with the affine space of achievable system responses \eqref{eq: constraint} is non-empty\footnote{This can be interpreted as a spatio-temporal generalization of controllability. Interested readers are referred to \cite{wang2019system,wang_localized_2014}, where the analogy is made formal.}.
\end{definition}

Although $d$-locality constraints are always convex subspace constraints, not all systems are $d$-localizable.  As we describe in Section \ref{sec:admm}, the locality diameter $d$ can be viewed as a design parameter, and for the remainder of the paper, we assume that there exists a $d<<N$ such that the system $(A,B)$ to be controlled is $d$-localizable.  We emphasize two facts.  We emphasize two facts. First, the parameter $d$ is tuned independently of the horizon $T$, and captures how ``far'' in the interconnection topology a disturbance striking a subsystem is allowed to spread -- as described in detail in \cite{wang2014localized,wang2016localized,wang2019system}, localized control can be thought of as a spatio-temporal generalization of deadbeat control.  Second, although beyond the scope of this paper, all of the presented results extend naturally to systems that are approximately localizable using a robust variant of the SLS parameterization described in \cite{matni_scalable_2017,anderson_system_2019}.

\begin{example}{\label{ex: chain implementation}}
Consider the chain in Example \ref{ex: chain graph}, and suppose that we enforce a $1$-locality constraint on the system responses: then $\mathbf \Phi_x$ is $1$-localized and $\mathbf \Phi_u$ is $2$-localized.  Due to the chain topology, this is equivalent to enforcing a tridiagonal structure on $\mathbf \Phi_x$ and a pentadiagonal structure on $\mathbf \Phi_u$.  The resulting $1$-outgoing and $2$-incoming sets at node $i$ are then given by $\textbf{out}_{i}(1)=\left\{i-1,i,i+1\right\}$ and $\textbf{in}_{i}(2)=\left\{i-2,i-1,i,i+1,i+2\right\}$, as illustrated in Figure \ref{fig: chain2}.
\begin{figure}[H]
    \centering
      \includegraphics[width=10cm, page=3]{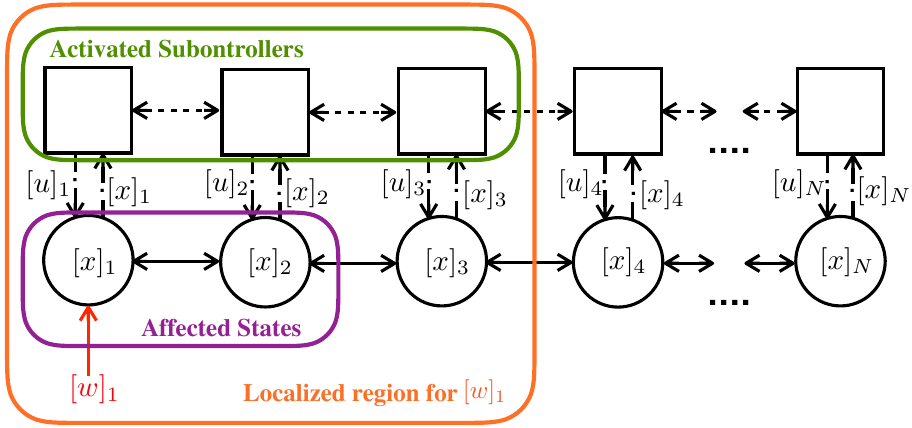}
    \caption{Graphical representation of the scenario described in Example \ref{ex: chain implementation}.}
    \label{fig: chain2}
\end{figure}
\end{example}

Finally, we introduce the necessary compatibility assumptions between the cost functions, state and input constraints, and $d$-local information exchange constraints.

\begin{assumption}{\label{assump: locality}}
The objective function $f_{t}$ in formulation (\ref{eq: MPC}) is such that $f_{t}(x)=\sum f_{t_{i}}([x]_{j\in\textbf{in}_{i}(d)},[u]_{j\in\textbf{in}_{i}(d)})$. The constrain sets in formulation (\ref{eq: MPC}) are such that $x\in\mathcal{X}=\mathcal{X}_1\times ... \times \mathcal{X}_n$, where $x \in \mathcal{X}$ if and only if $[x]_{j\in\textbf{in}_{i}(d)}\in\mathcal{X}_{i}$ for all $i$, and idem for $\mathcal{U}$.
\end{assumption}

Assumption \ref{assump: locality} imposes that whenever two subsystems are coupled through either the constraints or the objective function, they then must be within the $d$-local regions, as defined by their corresponding $d$-incoming and $d$-outgoing sets, of one another. This is a natural assumption for large structured networks where couplings between subsystems tend to occur at a local scale.  

We can formulate the DLMPC subproblem by incorporating locality constraints into the SLS MPC subproblem \eqref{eq: SLS}.
\begin{equation}\label{eq: SLS MPC}
\begin{array}{rl}
\underset{\mathbf{\Phi}_{x},\mathbf{\Phi}_{u}}{\text{min}} & \sum_{i=1}^N f^i([\mathbf{\Phi}_{x}x_0]_{j\in\textbf{in}_{i}(d)},[\mathbf{\Phi}_{u}x_0]_{j\in\textbf{in}_{i}(d)})\\
\text{s.t.} & [\mathbf{\Phi}_{x}x_0]_{j\in\textbf{in}_{i}(d)}\in\mathcal{X}_{i},\ [\mathbf{\Phi}_{u}x_0]_{j\in\textbf{in}_{i}(d)}\in\mathcal{U}_{i},\, i=1,\dots,N,\\
& Z_{AB}\mathbf{\Phi}=I,\, x_0 = x(t),\,\mathbf{\Phi}_{x},\mathbf{\Phi}_{u}\in\mathcal{L}_{d},  
\end{array}
\end{equation}
where the $f^i$ are defined so as to be compatible with the decomposition defined in Assumption \ref{assump: locality}.

 While it was not obvious how to impose locality constraints on information exchange in the original formulation of the MPC subproblem (\ref{eq: MPC}), it is straightforward to do so via the locality constraints $\mathbf{\Phi}_{x},\mathbf{\Phi}_{u}\in\mathcal{L}_{d}$ in formulation (\ref{eq: SLS MPC}). As these locality constraints are defined in terms of the $d$-hop incoming and outgoing sets of the interconnection topology of the system $\mathcal{G}_{(A,B)}$, the structure imposed on the system responses $\{\mathbf\Phi_u, \mathbf \Phi_x\}$ will be compatible with the structure of the matrix $Z_{AB}$ defining the affine constraint \eqref{eq: constraint}.  This structural compatibility in \emph{all optimization variables, cost functions, and constraints} is the key feature that we exploit to apply distributed optimization techniques in the next section to scalably and \emph{exactly} solve problem \eqref{eq: SLS MPC}.
 
 \begin{remark} Note that although $d$-locality constraints can always be imposed as convex subspace constraints, not all systems are $d$-localizable.  As we describe in the sequel, the locality diameter $d$ can be viewed as design parameter, and for the remainder of the paper, we assume that there exists a $d<<N$ such that the system $(A,B)$ to be controlled is $d$-localizable.  Although beyond the scope of this paper, all of the presented results extend naturally to systems that are approximately localizable using a robust variant of the SLS parameterization described in \cite{matni_scalable_2017,anderson_system_2019}.
\end{remark}
 
We end by commenting briefly as to why previous methods \cite{goulart_optimization_2006,furieri_robust_2017} do not enjoy this feature.  We focus on the method defined in \cite{goulart_optimization_2006}, as a similar argument applies to the synthesis problem in \cite{furieri_robust_2017}.  Intuitively, the disturbance based feedback parameterization of \cite{goulart_optimization_2006} only parameterizes the closed loop map $\mathbf \Phi_u$ from $\mathbf w \to \mathbf u$, and leaves the state $\mathbf x$ as a free variable.  This can be made explicit by noticing that the disturbance feedback parameterization of \cite{goulart_optimization_2006} can be recovered from the SLS parameterization of Theorem \ref{thm: SLS} by multiplying the affine constraint \eqref{eq: constraint} by $\mathbf{w}$ on the right, and setting $\mathbf{x} = \mathbf \Phi_x \mathbf w$. Further, this immediately implies, by the result of \cite{wang2016localized}, that the SLS MPC subproblem \eqref{eq: SLS} is equivalent to the affine problem \eqref{eq: SLS MPC} when restricted to solving over linear-time-varying feedback policies. Setting $\mathbf{w} = [x_0^\top,0,\dots,0]^\top$, and relabeling the corresponding decision variables in the SLS MPC subproblem \eqref{eq: SLS} yields 
\begin{equation}\label{eq: goulart}
\begin{array}{rl}
\underset{\mathbf{x},\mathbf{\Phi}_{u}}{\text{min}} & f(\mathbf{x},\mathbf{\Phi}_{u}[0]x_{0})\\
 \text{s.t.} &  (I - Z\hat{A})\mathbf{x}=Z\hat{B}\mathbf{\Phi}_u[0] x_0 + E_1x_0,\,\\
 &x_0 = x(t),\, \mathbf{x}\in \mathcal{X}^T, \mathbf{\Phi}_{u}[0]x_{0} \in\mathcal{U}^T,
\end{array}
\end{equation}
where $E_1$ is a block-column matrix with first block element set to identity, and all others set to 0.  This is a special case of the optimization problem over disturbance feedback policies suggested in Section 4 of \cite{goulart_optimization_2006} with only a nonzero initial condition.  Notice that regardless as to what structure is imposed on the objective functions, constraints, and the map $\mathbf{\Phi}_u$, the resulting optimization problem is strongly and globally coupled by the affine constraint $(I - Z\hat{A})\mathbf{x}=Z\hat{B}\mathbf{\Phi}_u x_0 + E_1x_0$, because the state variable $\mathbf{x}$ is always dense.  A similar coupling arises in the Youla based parameterization suggested in \cite{furieri_robust_2017}.  In contrast, by explicitly parameterizing the additional system response $\mathbf{\Phi}_x$ from process noise to state, i.e., from $\mathbf{w}\to\mathbf{x}$, we can naturally enforce the structure needed for distributed optimization techniques to be fruitfully applied.  

\section{An ADMM based Distributed AND Localized Solution}\label{sec:admm}

We start with a brief overview of the ADMM algorithm, and then show how it can be used to decompose the DLMPC sub-problem (\ref{eq: SLS MPC}) into sub-problems that can be solved using only $d$-local information. For the sake of clarity, we introduce a simpler version of the algorithm first where only dynamical coupling is considered, which we then extend to the constraints and objective functions that introduce $d$-localized couplings, as defined in Assumption \ref{assump: locality}.

\subsection{The Alternating Direction Method of Multipliers}

The ADMM \cite{boyd_distributed_2010} algorithm has proved successful solving large-scale optimization problems that respect a certain partial-separability structure. In particular, given the following optimization problem
\begin{equation}\label{eq: optimization}
\begin{aligned}
& \underset{x,y}{\text{min}} &  &f(x)+g(y)\\
& \ \text{s.t.} &  &Ax+By=c,
\end{aligned}
\end{equation}
ADMM - in its scaled form - solves (\ref{eq: optimization}) with the following update rules:

\begin{equation}\label{eq: ADMM}
\begin{aligned}
& x^{k+1} = \arg\min_x f(x)+\frac{\rho}{2}\left\| Ax+By^{k}-c+z^{k}\right\|_{2}^{2}\\
& y^{k+1} = \arg\min_y g(y)+\frac{\rho}{2}\left\| Ax^{k+1}+By-c+z^{k}\right\|_{2}^{2}\\
& z^{k+1} = z^{k}+Ax^{k+1}+By^{k+1}-c.
\end{aligned}
\end{equation}
    
ADMM is particularly powerful when the iterate sub-problems defined in equation (\ref{eq: ADMM}) can be solved in closed form, which is the case in many practically relevant cases \cite{boyd_distributed_2010}, as this allows for rapid execution and convergence of the algorithm.  
Further, under mild assumptions ADMM enjoys strong and general convergence guarantees. 

\begin{theorem}\label{thm: ADMM}
Assume that extended real valued functions $f:\mathbb{R}^n\rightarrow \mathbb{R}\cup\left\{+\infty \right\}$ and $g:\mathbb{R}^m\rightarrow \mathbb{R}\cup\left\{+\infty \right\}$ are closed, proper, and convex. Moreover, assume that the unaugmented Lagrangian has a saddle point. Then, the ADMM iterates in equation (\ref{eq: ADMM}) satisfy the following:
\begin{itemize}
    \item \textbf{Residual convergence:} $r_{t}\rightarrow 0$ as $t\rightarrow \infty$, i.e. the iterates approach feasibility.
    \item \textbf{Objective convergence:} $f(x_{t})+g(z_{t})\rightarrow p^{*}$ as $t\rightarrow 0$, i.e. the objective function of the iterates approaches the optimal value.
    \item \textbf{Dual variable convergence:} $y_{t}\rightarrow y^{*}$ as $t\rightarrow \infty$, where $y^{*}$ is a dual optimal point.
\end{itemize}
\end{theorem}
    
Hence, we impose the following additional assumptions so as to ensure that Theorem \ref{thm: ADMM} is applicable to the considered problem.

\begin{assumption}\label{assump: feasible solution}
Problem (\ref{eq: SLS MPC}) has a feasible solution in the relative interior of $\mathcal{X}^T$ and $\mathcal{U}^T$.
\end{assumption} 
\begin{assumption}\label{assump: saddle point}
The constraint sets $\mathcal{X}^T$ and $\mathcal{U}^T$ in formulation (\ref{eq: SLS MPC}) are closed and convex. The objective function $f(\mathbf{\Phi}x_0)$ is a closed, proper, and convex function for all choices of $x_0\neq 0$.
\end{assumption}
            
\subsection{DLMPC algorithm}  

Now we present the proposed algorithm for the localized synthesis of the DLMPC problem (\ref{eq: SLS MPC}).

\subsubsection{DLMPC without Coupling Constraints}  

In this subsection we present here a simplified version of the algorithm that contains its main features: namely, we illustrate how to decompose the MPC subroutine (\ref{eq: SLS MPC}) by exploiting its separability. We begin by restricting ourselves to the case where neither the objective function nor the constraints introduce any coupling between subsystems. In the next subsection the general version of the algorithm will be introduced.
We show that the DLMPC problem (\ref{eq: SLS MPC}) can be decomposed into local subproblems that can be solved using only $d$-local information and system models.  In particular, consider the DLMPC subproblem solved at time $t$:
\begin{equation}\label{eq: MPC without ADMM}
\begin{array}{rl}
\underset{\mathbf{\Phi}}{\text{min}} & f(\mathbf{\Phi}x_0)\\
 \text{s.t.} &  Z_{AB}\mathbf{\Phi} = I,\ \mathbf{\Phi}x_0\in\mathcal{P},\ \mathbf{\Phi}\in\mathcal{L}_{d}, \, x_0 = x(t),
\end{array}
\end{equation}
where we let $Z_{AB}:=\left[I-Z\hat A\ \ -Z\hat B\right]$, $\mathbf{\Phi}:=[\mathbf{\Phi}_{x}^{\mathsf{T}}\ \mathbf{\Phi}_{u}^{\mathsf{T}}]^{\mathsf{T}}$, and $\mathbf{\Phi}x_0\in\mathcal{P}\iff\mathbf{\Phi}_{x}x_0 \in\mathcal{X}^T$\ and $\mathbf{\Phi}_{u}x_0\in\mathcal{U}^T$. Moreover, according to Assumption \ref{assump: locality}, $f$ can be written as $f(\mathbf{x},\mathbf{u})=\sum f_{i}([\mathbf{x}]_{i},[\mathbf{u}]_{i})$, or equivalently, that $f(\mathbf{\Phi}x_0)=\sum f_{i}(\mathbf{\Phi}(\mathfrak{r}_i,:)x_0)$ since $\begin{bmatrix} [\mathbf{x}]_{i} \\ [\mathbf{u}]_{i}  \end{bmatrix}=\mathbf{\Phi}(\mathfrak{r}_i,:)x_0$ where $\mathfrak{r}_i$ is the set of rows in $\mathbf{\Phi}$ corresponding to subsystem $i$. Similarly, the constraints must satisfy that $\mathbf{x}\in\mathcal{X} =\mathcal{X}_1\times ... \times \mathcal{X}_n$, where each $[\mathbf{x}]_{i}\in\mathcal{X}_{i}$, and idem for $\mathcal{U}$, implying that $\mathbf{\Phi}(\mathfrak{r}_i,:)x_0\in\mathcal{P}_{i}$. These notions are formalized through a slight modification of the definitions in \cite{wang_separable_2018}:

\begin{definition}
The functional $g(\mathbf{\Phi})$ is \textit{column-wise separable} with respect to the partition $\mathfrak{c}=\left\{\mathfrak{c}_{1}, . . . , \mathfrak{c}_{p}\right\}$ if it can be written as $g(\mathbf{\Phi})= \sum_{j=1}^{p} g_{j}(\mathbf{\Phi}(:,\mathfrak{c}_{j}))$ for some functionals $g_{j}$ for $j = 1,...,p$. Equivalently, $g(\mathbf{\Phi})$ is \textit{row-wise separable} with respect to the partition $\mathfrak{r}$ if it can be written as $g(\mathbf{\Phi})= \sum_{j=1}^{p} g_{j}(\mathbf{\Phi}(\mathfrak{r}_{j},:)$.
\end{definition}
\begin{definition} A constraint-set $\mathcal{P}$ is \textit{column-wise separable} with respect to the partition $\mathfrak{c}=\left\{\mathfrak{c}_{1}, . . . , \mathfrak{c}_{p}\right\}$ when $\mathbf{\Phi}\in\mathcal{P}\iff\mathbf{\Phi}(:,\mathfrak{c}_{j})\in\mathcal{P}_{j}\ \text{for } j=1,...,p$ is satisfied for some sets $\mathcal{P}_{j}$ for $j = 1,...,p$. 
Equivalently, $\mathcal{P}$ is \textit{row-wise separable} with respect to the partition $\mathfrak{r}$ if  $\mathbf{\Phi}\in\mathcal{P}\iff\mathbf{\Phi}(\mathfrak{r}_{j},:)\in\mathcal{P}_{j}\ \text{for } j=1,...,p$.
\end{definition}

If all objective functions and constraints in an optimization problem are column-wise (row-wise) separable with respect to a partition $\mathfrak{c}$ ($\mathfrak r$) of cardinality $p$, then the optimization problem trivially decomposes into $p$ independent subproblems.  However, while Assumption \ref{assump: locality} imposes that $f_{x_0}(\mathbf{\Phi})=f(\mathbf{\Phi}x_0)$ and $\mathcal{P}$ are row-wise separable in the optimization variable $\mathbf{\Phi}$, the achievability constraint $Z_{AB}\mathbf{\Phi} = I$ is column-wise separable in the optimization variable $\mathbf{\Phi}$.  As we show next, this partially-separable structure can be exploited within an ADMM based algorithm to reduce each ADMM iterate subproblem \eqref{eq: ADMM} to a row or column-wise separable optimization problem, allowing the algorithm to decompose and be solved at scale.
\begin{definition}{\label{partially separable}}
An optimization problem is partially separable if it can be written as
\begin{equation}\label{eq: partially separable}
 \underset{\mathbf{\Phi}}{\text{min }} g^{(r)}(\mathbf{\Phi})+g^{(c)}(\mathbf{\Phi}) \ \text{s.t. } 
     \mathbf{\Phi}\in\mathcal{S}^{(r)}\cap\mathcal{S}^{(c)},
\end{equation}
for row-wise separable $g^{(r)}$ and $\mathcal{S}^{(r)}$, and column-wise separable $g^{(c)}$ and $\mathcal{S}^{(c)}$.
\end{definition}

Since problem (\ref{eq: MPC without ADMM}) is partially separable, we reformulate the DLMPC subproblem \eqref{eq: SLS MPC} so that is of the form \eqref{eq: partially separable}:
\begin{equation} \label{eq: MPC duplicated variable}
\begin{aligned}
& \underset{\mathbf{\Phi},\mathbf{\Psi}}{\text{min}} && f(\mathbf{\Phi}x_0)\\
& \ \text{s.t.} &  &\begin{aligned} 
     &Z_{AB}\mathbf{\Psi} = I,\ \mathbf{\Phi}x_0\in\mathcal{P},\ \left\{\mathbf{\Phi},\mathbf{\Psi}\right\}\in\mathcal{L}_{d}, \, \mathbf{\Phi}=\mathbf{\Psi}.
\end{aligned}
\end{aligned}
\end{equation}
            
By duplicating the decision variable, we can decompose the DLMPC subproblem \eqref{eq: SLS MPC} into a column-wise separable iterate subproblem in $\mathbf{\Phi}$, and a row-wise seperable iterate subproblem in $\mathbf{\Psi}$ -- thus problem \eqref{eq: MPC duplicated variable} is partially separable, and is amenable to a distributed solution via ADMM:
\begin{subequations}\label{eq: MPC ADMM global 1}
\begin{align}
& \mathbf{\Phi}^{k+1} = 
\left\{\begin{aligned}
&\underset{\mathbf{\Phi}}{\text{argmin}} && f(\mathbf{\Phi}x_0)+\frac{\rho}{2}\left\Vert\mathbf{\Phi}-\mathbf{\Psi}^{k}+\mathbf{\Lambda}^{k}\right\Vert^{2}_{F}\\
&\text{s.t.} && \begin{aligned}
     &\mathbf{\Phi}x_0\in\mathcal{P},\ \mathbf{\Phi}\in\mathcal{L}_{d}\\
\end{aligned}
\end{aligned}\right\} \label{eq: MPC ADMM global 1 - row}
\\[5pt]
& \mathbf{\Psi}^{k+1} = 
\left\{\begin{aligned}
&\underset{\mathbf{\Psi}}{\text{argmin}} && \left\Vert\mathbf{\Phi}^{k+1}-\mathbf{\Psi}+\mathbf{\Lambda}^{k}\right\Vert^{2}_{F}\\
&\text{s.t.} && \begin{aligned}
     &Z_{AB}\mathbf{\Psi} = I,\ \mathbf{\Psi}\in\mathcal{L}_{d}\\
\end{aligned}
\end{aligned}\right\}\label{eq: MPC ADMM global 1 - column}
\\[5pt]
& \mathbf{\Lambda}^{k+1} = \mathbf{\Lambda}^{k}+\mathbf{\Phi}^{k+1}-\mathbf{\Psi}^{k+1}. \label{eq: MPC ADMM global 1 - lagrange}\end{align}
\end{subequations}            

The squared Frobenius norm is both row-wise and column-wise separable. Therefore, the resulting iterate subproblems in (\ref{eq: MPC ADMM global 1}) are separable: iterate subproblem (\ref{eq: MPC ADMM global 1 - row}) is row-wise separable with respect to the row partition $\mathfrak{r}$ induced by the subsystem-wise partitions of the state and control inputs, $[x]_i$ and $[u]_i$, iterate subproblem (\ref{eq: MPC ADMM global 1 - column}) is column-wise separable with respect to the column partition induced in a analogous manner, and iterate subproblem (\ref{eq: MPC ADMM global 1 - lagrange}) is component-wise separable. Hence, each of the iterate subproblems in (\ref{eq: MPC ADMM global 1}) can be decomposed into column, row, or element-wise subproblems that can solved independently and in parallel, with each sub-controller $i$ computing the solution to its component of the row or column-wise partition.

 Moreover, by enforcing that the system responses be $d$-localized, i.e., that $\mathbf{\Phi}_{x},\mathbf{\Phi}_{u}\in\mathcal{L}_d$, the resulting subproblem variables are sparse, allowing for a significant reduction in the dimension of the local subproblem.  For example, when considering the column-wise subproblem evaluated at 
 subsystem $j$, the $i$th row of the $j$th subsystem column partitions of $\mathbf{\Phi}_x(:,\mathfrak{c}_{j})$ and $\mathbf{\Phi}_u(:,\mathfrak{c}_{j})$) is nonzero only if  $i\in\cup_{k\in\textbf{out}_j(d)}\mathfrak{r}_k$ and $i\in\cup_{k\in\textbf{out}_j(d+1)}\mathfrak{r}_k$, respectively. 

In particular, sub-controller $i$ solves the subproblems:
 \begin{subequations}\label{eq: MPC ADMM localized 1}
\begin{align}
& [\mathbf{\Phi}]_{i_{r}}^{k+1} = \left\{
\begin{aligned}
&\underset{\mathbf{[\mathbf{\Phi}]}_{i_{r}}}{\text{argmin}} &&  f_{i}([\mathbf{\Phi}]_{i_{r}}[x_0]_{i_{r}})+\frac{\rho}{2}\left\Vert[\mathbf{\Phi}]_{i_{r}}-[\mathbf{\Psi}]_{i_{r}}^{k}+[\mathbf{\Lambda}]_{i_{r}}^{k}\right\Vert^{2}_{F}\\
&\text{s.t.} && \begin{aligned}
     &[\mathbf{\Phi}]_{i_{r}}[x_0]_{i_{r}}\in\mathcal{P}_i\\
\end{aligned}
\end{aligned}\right\} \label{eq: MPC ADMM localized 1 - row}
\\
& [\mathbf{\Psi}]_{i_{c}}^{k+1} = 
\left\{\begin{aligned}
&\underset{[\mathbf{\Psi}]_{i_{c}}}{\text{argmin}} &&  \left\Vert[\mathbf{\Phi}]_{i_{c}}^{k+1}-[\mathbf{\Psi}]_{i_{c}}+[\mathbf{\Lambda}]_{i_{c}}^{k}\right\Vert^{2}_{F}\\
&\text{s.t.} && \begin{aligned}
     &[Z_{AB}]_{i_{c}}[\mathbf{\Psi}]_{i_{c}} = [I]_{i_{c}} \\
\end{aligned}
\end{aligned}\right\}\label{eq: MPC ADMM localized 1 - column}
\\[10pt]
& [\mathbf{\Lambda}]_{i_{r}}^{k+1} =[\mathbf{\Lambda}]_{i_{r}}^{k}+[\mathbf{\Phi}]_{i_{r}}^{k+1}-[\mathbf{\Psi}]_{i_{r}}^{k+1}, \label{eq: MPC ADMM localized 1 - lagrange}\end{align}
\end{subequations} 
where to lighten notational burden, we let $[\mathbf{\Phi}]_{i_{r}}:=\mathbf{\Phi}(\mathfrak{s_{\mathfrak{r}_{i}}},\mathfrak{r}_{i})$, where the set $\mathfrak{r}_{i}$ represents the set of rows that the controller $i$ is solving for, and the set $\mathfrak{s_{\mathfrak{r}_{i}}}$ is the set of columns associated to the rows in $\mathfrak{r}_{i}$ by the locality constraints $\mathcal{L}_{d}$. 
An equivalent argument applies to
$[\mathbf{\Phi}]_{i_{c}}:=\mathbf{\Phi}(\mathfrak{c}_{i},\mathfrak{s_{\mathfrak{c}_{i}}})$ where the set $\mathfrak{c}_{i}$ represents the set of columns that the controller $i$ is solving for, and the corresponding set $\mathfrak{s_{\mathfrak{c}_{i}}}$ is the set of columns associated to the rows in $\mathfrak{c}_{i}$. For example, when considering the row-wise subproblem \eqref{eq: MPC ADMM localized 1 - row} evaluated at subsystem $i$, the $j$th column of the $i$th subsystem row partition $\mathbf{\Phi}_x(\mathfrak{r}_{i},:)$ and $\mathbf{\Phi}_u(\mathfrak{r}_{i},:)$ is nonzero only if  $j\in\textbf{in}_j(d)$ and $j\in\textbf{in}_j(d+1)$, respectively.  It follows that subsystem $i$ only requires a corresponding subset of the local sub-matrices $[A]_{k,\ell}, [B]_{k,\ell}$ to solve its respective subproblem. All column/row/matrix subsets described above can be found algorithmically (see Appendix A of \cite{wang_separable_2018}). 

\begin{remark}
Problem (\ref{eq: MPC ADMM localized 1 - column}) can be solved in closed form:
\begin{equation}{\label{eq: MPC ADMM localized 1 - proxi}}
[\mathbf{\Psi}]_{i_{c}}^{k+1} = \big([\mathbf{\Phi}]_{i_{c}}^{k+1}+[\mathbf{\Lambda}]_{i_{c}}^{k}\big)+[Z_{AB}]_{i_{c}}^{+}\Big([I]_{i_{c}} - [Z_{AB}]_{i_{c}}\big([\mathbf{\Phi}]_{i_{c}}^{k+1}+[\mathbf{\Lambda}]_{i_{c}}^{k}\big)\Big),
\end{equation}
where $[Z_{AB}]_{i_{c}}^{+}$ denotes the pseudo-inverse of $[Z_{AB}]_{i_{c}}$. This pseudo-inverse can be computed once off-line, reducing the evaluation of update step \eqref{eq: MPC ADMM localized 1 - proxi} to matrix multiplication.
\end{remark}{}

Notice that in general, in general $\mathfrak r_{i}\subset \mathfrak s_{\mathfrak c_{i}}$ and  $\mathfrak c_{i}\subset \mathfrak s_{\mathfrak r_{i}}$. Hence, each subsystem $i$ is computing updates for the sub-matrix $\mathbf{\Phi}(\mathfrak s_{\mathfrak r_{i}},c_{i})$ and the sub-matrix $\mathbf{\Phi}(\mathfrak r_i, s_{\mathfrak c_{i}})$ of the global system response variables $\mathbf{\Phi}$ and $\mathbf{\Psi}$. In particular, for subsystem $i$ to solve its local iterate subproblems (\ref{eq: MPC ADMM localized 1}), information sharing among subsystems is needed. However, as we impose $d$-locality constraints on the system responses, information only needs to be collected from d-hop neighbors. Similarly, only a $d$-local subset initial condition $x_0=x(t)$ is needed to solve the local iterate subproblems (\ref{eq: MPC ADMM localized 1}).

Algorithm \ref{alg: I} summarizes the implementation at subsystem $i$ of the ADMM based solution to the DLMPC subproblem (\ref{eq: SLS MPC}).  In the final step of Algorithm \ref{alg: I}, we let $[x_0]_{\mathfrak s_{\mathfrak{r_i}}}$ denote the subset of elements of $x_0$ associated with the columns in $\mathfrak s_{\mathfrak{r_i}}$, such that $[\Phi_u^{0,0}]_{i_{r}}x_0 =[\Phi_u^{0,0}]_{i_{r}}[x_0]_{\mathfrak s_{\mathfrak{r_i}}}$ \footnote{Recall that per in the notation section, $\Phi_u^{0,0}$ represents the upper-most left-most matrix in the block-diagonal operator $\mathbf{\Phi}_u$.}.  Algorithm 1 is run in parallel by each sub-controller, and makes clear that only $d$-local information and system models are needed to solve the ADMM iterate subproblems \eqref{eq: MPC ADMM localized 1} at subsystem $i$.

\setlength{\textfloatsep}{0pt}
\begin{algorithm}[ht]
\caption{Subsystem $i$ DLMPC implementation}\label{alg: I}
\begin{algorithmic}[1]
\State \textbf{input:} convergence tolerance parameters $\epsilon_p>0$, $\epsilon_d>0$
\State Measure local state $[x(t)]_{i}$.
\State Share the measurement with $\textbf{out}_{i}(d)$.
\State Solve optimization problem (\ref{eq: MPC ADMM localized 1 - row}).
\State Share $[\mathbf{\Phi}]_{i_{r}}^{k+1}$ with $\textbf{out}_{i}(d)$. Receive the corresponding $[\mathbf{\Phi}]_{j_{r}}^{k+1}$ from $\textbf{in}_{i}(d)$ and build $[\mathbf{\Phi}]_{i_{c}}^{k+1}$.
\State Solve optimization problem (\ref{eq: MPC ADMM localized 1 - column}) via the closed form solution (\ref{eq: MPC ADMM localized 1 - proxi}).
\State Share $[\mathbf{\Psi}]_{i_{c}}^{k+1}$ with $\textbf{out}_{i}(d)$. Receive the corresponding $[\mathbf{\Psi}]_{j_{c}}^{k+1}$ from $\textbf{in}_{i}(d)$ and build $[\mathbf{\Psi}]_{i_{r}}^{k+1}$.
\State Perform the multiplier update step (\ref{eq: MPC ADMM localized 1 - lagrange}).
\State Check convergence as $\left\Vert[\mathbf{\Phi}]_{i_{r}}^{k+1}-[\mathbf{\Psi}]_{i_{r}}^{k+1}\right\Vert_F\leq\epsilon_{p}$ and $\left\Vert[\mathbf{\Psi}]_{i_{r}}^{k+1}-[\mathbf{\Psi}]_{i_{r}}^{k}\right\Vert_F\leq\epsilon_{d}$. 
\State If converged, apply computed control action $[u_0]_i = [\Phi_u^{0,0}]_{i_{r}}[x_0]_{\mathfrak s_{\mathfrak{r_i}}}$, and return to 2, otherwise return to 4. 
\end{algorithmic}
\end{algorithm}

\subsubsection*{Computational complexity of the algorithm}         
The computational complexity of the algorithm is determined by update steps 4, 6 and 8. In particular, steps 6 and 8 can be directly solved in closed form, reducing their evaluation to the multiplication of matrices of dimension $O(d^{2}T)$. In certain cases, step 4 can also be computed in closed form if a proximal operator exists for the formulation. For instance this is true if it reduces to quadratic convex cost function subject to affine equality constraints. Regardless, each local iterate sub-problem is over $O(d^{2}T)$ optimization variables subject to $O(dT)$ constraints, leading to a significant computational saving when $d<<N$. The communication complexity - as determined by steps 3, 5 and 7 - is limited to the local exchange of information between $d$-local neighbors. 

\subsubsection*{Convergence of the algorithm}

One can show convergence by leveraging Theorem \ref{thm: ADMM}.

\begin{corollary}\label{corollary}
Algorithm \ref{alg: I} satisfies residual convergence, objective convergence and dual variable convergence as defined in Theorem \ref{thm: ADMM}.
\end{corollary}
\begin{proof}
Algorithm \ref{alg: I} is built upon algorithm (\ref{eq: MPC ADMM localized 1}), which is merely algorithm (\ref{eq: MPC ADMM global 1}) after exploiting locality. Thus to prove Corollary \ref{corollary} we only need to show that the ADMM algorithm (\ref{eq: MPC ADMM global 1}) satisfies the assumptions in Theorem \ref{thm: ADMM}. 

Define the extended-real-value functional $h(\mathbf{\Phi})$ by
\begin{equation*}
    h(\mathbf{\Phi})=\begin{cases}
    \mathbf{\Phi}x_0 &\text{if $Z_{AB}\mathbf{\Phi}=I,\ \mathbf{\Phi}x_0\in\mathcal{P},\ \mathbf{\Phi}\in\mathcal{L}_{d}$}\\ 
    \infty &\text{otherwise}. \end{cases}
\end{equation*}
 The constrained optimization in (\ref{eq: MPC duplicated variable}) can equivalently be written in terms of $h(\mathbf{\Phi})$ with the constraint $\mathbf{\Phi}=\mathbf{\Psi}$. 
\begin{equation*}
\begin{aligned}
& \underset{\mathbf{\Phi},\mathbf{\Psi}}{\text{min}} && h(\mathbf{\Phi}) \ \text{s.t.} &    &\mathbf{\Phi}=\mathbf{\Psi}.
\end{aligned}
\end{equation*}
Notice that by Assumption \ref{assump: saddle point}, $f(\mathbf{\Phi}x_0)$ is closed, proper, and convex, and $\mathcal{P}$ is a closed and convex set. Moreover, the remaining constraints $Z_{AB}\mathbf{\Phi}=I$ and $\mathbf{\Phi}\in\mathcal{L}_{d}$ are also closed and convex. Hence, $h(\mathbf{\Phi})$ is closed, proper, and convex. 
It only remains to show that the Lagrangian has a saddle point. This condition is equivalent to showing that strong duality holds \cite{boyd_convex_2004}. By Assumption \ref{assump: feasible solution}, Slater's condition is automatically satisfied, and therefore the Lagrangian of the problem has a saddle point.
Since both conditions of Theorem \ref{thm: ADMM} are satisfied, the ADMM algorithm in (\ref{eq: MPC ADMM global 1}) satisfies residual convergence, objective convergence and dual variable convergence as defined in Theorem \ref{thm: ADMM}. Since Algorithm \ref{alg: I} results from leveraging (\ref{eq: MPC ADMM global 1}), guaranteeing convergence of algorithm (\ref{eq: MPC ADMM global 1}) automatically guarantees convergence of Algorithm \ref{alg: I}.
\end{proof}

\subsubsection*{Recursive feasibility and stability}

In order to guarantee recursive feasibility, we can make the standard assumption that the horizon $N$ is chosen sufficiently long (Corollary 13.2, \cite{borrelli_predictive_2017}).  As the complexity of the subproblems now scales with the localized radius $d$, choosing a longer horizon $N$ no longer represents as substantial a computational burden.  Similarly, in order to guarantee stability, a simple sufficient condition is to set the terminal constraint $[x_N]_i=0$. Although this is a conservative sufficient condition, and we leave exploring the development of more principled terminal costs and constraint sets as well as integrating robustness to additive disturbance to future work, the ability of SLS to synthesize localized distributed controllers that define localized forward invariant sets satisfying state and input constraints \cite{chen_system_2019} via distributed optimization offers a promising avenue forward.

\subsubsection*{Communication dropouts}

In real-life applications, communication lost due to data losses or corrupted communications links can occur. Given that the presented algorithm relies on multiple exchanges of information between the subsystems, how communication lost affects the closed-loop performance of the algorithm is an interesting question. Although a formal analysis is left as future research, the work done in \cite{li_robust_2018} illustrates that it would be possible to slightly modify the proposed ADMM-based scheme to make it robust to unreliable communication links. Furthermore, extensions of ADMM to deal with data privacy and leakage when agents exchange information such as the ones carried out in \cite{ding_differentially_2019} can be directly applicable to our framework for sensitive-data applications. Lastly, the impact of the network topology in the convergence and performance of ADMM as studied in \cite{franca_how_2017} sets the ground for extending this work to time-changing topologies.

\subsubsection{DLMPC subject to localized coupling constraints}
Building on Algorithm \ref{alg: I}, we now show how DLMPC can be extended to allow for coupling between subsystems by the constraints and objective function, so long as the coupling is compatible with the $d$-localized constraints being imposed, i.e., so long as the objective function and constraints satisfy Assumption \ref{assump: locality}.
Consider the DLMPC sub-problem \eqref{eq: MPC without ADMM}, where the objective function and constraints satisfy the locality properties imposed by Assumption \ref{assump: locality}.  Due to this local coupling, the problem is no longer partially separable -- however, we may still rewrite it as
\begin{equation}\label{eq: MPC without ADMM - X trick}
\begin{aligned}
& \underset{\mathbf{X},\mathbf{\Phi}}{\text{min}} && f(\mathbf{X})\\
& \ \text{s.t.} &  &\begin{aligned} 
     &\mathbf{X}=\mathbf{\Phi}x_0,\ \\
     &Z_{AB}\mathbf{\Phi} = I \\
     &\mathbf{X}\in\mathcal{P},\ \mathbf{\Phi}\in\mathcal{L}_{d}.
\end{aligned}
\end{aligned}
\end{equation}

The first constraint is row-wise separable for $\mathbf{\Phi}$ while the second is column-wise separable in $\mathbf{\Phi}$. Applying the same variable duplication process as above and applying ADMM yields
\begin{subequations}\label{eq: MPC ADMM global 2}
\begin{align}
& [\mathbf{\Phi}^{k+1},\ \mathbf{X}^{k+1}] = 
\left\{\begin{aligned}
&\underset{\mathbf{\Phi},\mathbf{X}}{\text{argmin}} && f(\mathbf{X})+\frac{\rho}{2}\left\Vert\mathbf{\Phi}-\mathbf{\Psi}^{k}+\mathbf{\Lambda}^{k}\right\Vert^{2}_{F}\\
&\text{s.t.} && \begin{aligned}
     &\mathbf{X}=\mathbf{\Phi}x_0,\ \mathbf{X}\in\mathcal{P},\ \mathbf{\Phi}\in\mathcal{L}_{d}\\
\end{aligned}
\end{aligned}\right\} \label{eq: MPC ADMM global 2 - row}
\\[0pt]
& \mathbf{\Psi}^{k+1} = 
\left\{\begin{aligned}
&\underset{\mathbf{\Psi}}{\text{argmin}} && \left\Vert\mathbf{\Phi}^{k+1}-\mathbf{\Psi}+\mathbf{\Lambda}^{k}\right\Vert^{2}_{F}\\
&\text{s.t.} && \begin{aligned}
     &Z_{AB}\mathbf{\Psi} = I,\ \mathbf{\Psi}\in\mathcal{L}_{d}\\
\end{aligned}
\end{aligned}\right\}\label{eq: MPC ADMM global 2 - column}
\\[5pt]
& \mathbf{\Lambda}^{k+1} = \mathbf{\Lambda}^{k}+\mathbf{\Phi}^{k+1}-\mathbf{\Psi}^{k+1} \label{eq: MPC ADMM global 2 - lagrange}\end{align}
\end{subequations} 

While the iterate sub-problems \eqref{eq: MPC ADMM global 2 - column} and \eqref{eq: MPC ADMM global 2 - lagrange} enjoy column-wise and element-wise separability, iterate sub-problem \eqref{eq: MPC ADMM global 2 - row} is subject to local coupling due to the objective function $f$ and constraint $\mathbf X \in \mathcal P$.
In order to solve sub-problem iterate \eqref{eq: MPC ADMM global 2 - row} in a manner that respects the $d$-localized communication constraints, we propose an ADMM based consensus-like algorithm, similar to that used in \cite{costantini_decomposition_2018}. Hence, the solution to iterate sub-problem (\ref{eq: MPC ADMM global 2 - row}) is obtained by having each subsystem $i$ solve
\begin{subequations}\label{eq: MPC ADMM consensus}
\begin{align}
& [[\mathbf{\Phi}]_{i_{r}}^{k+1,n+1},\ [\mathbf{X}]_{i_{s}}^{n+1}] =
\left\{\begin{aligned}
& &&\underset{[\mathbf{\Phi}]_{i_{r}},[\mathbf{X}]_{i_{s}}}{\text{argmin}} && \begin{aligned} &f_{i}(\mathbf{X})+\frac{\rho}{2}\left\Vert[\mathbf{\Phi}]_{i_{r}}-[\mathbf{\Psi}]_{i_{r}}^{n}+[\mathbf{\Lambda}]_{i_{r}}^{n}\right\Vert^{2}_{F}\\
&+\frac{\mu}{2}\underset{j\in \textbf{in}_{i}(d)}{\sum}\left\Vert[\mathbf{X}]_{j}^{n+1}-[\mathbf{Z}]_{i}+[\mathbf{Y}]_{ij}^{n}\right\Vert^{2}_{F}\end{aligned}\\
&\ &&\ \text{s.t.} && \begin{aligned}
     &[\mathbf{X}]_{i}=[\mathbf{\Phi}]_{i_{r}}[x_0]_{i_{r}},\ [\mathbf{X}]_{i_{s}}\in\mathcal{P}_i\\
\end{aligned}
\end{aligned} \right\}\label{eq: MPC ADMM consensus 1}
\\[0pt]
&[\mathbf{Z}]_{i}^{n+1} = \frac{1}{\vert\textbf{in}_{i}(d)\vert}\underset{j\in \textbf{in}_{i}(d)}{\sum}\left\Vert[\mathbf{X}]_{j}^{n+1}+[\mathbf{Y}]_{ij}^{n}\right\Vert^{2}_{F}\label{eq: MPC ADMM consensus 2}
\\[5pt]
& \mathbf{[Y]}_{ij}^{n+1} = \mathbf{[Y]}_{ij}^{n}+\mathbf{[X]}_{i}^{n+1}-\mathbf{[Z]}_{j}^{n+1}, \label{eq: MPC ADMM consensus 3}\end{align}
\end{subequations} 
where the notation used is as in \eqref{eq: MPC ADMM localized 1}. In particular $[\mathbf{X}]_{i_{s}}$ is the concatenation of components $[\mathbf{X}]_j$ satisfying $j \in \mathbf{in}_{i}(d)$, whereas $[\mathbf{X}]_{i}$ is restricted to only those components of $[\mathbf X]_{i_s}$ needed by subsystem $i$ to compute $[\mathbf{\Phi}]_{i_{r}}x_0 =[\mathbf{\Phi}]_{i_{r}}[x_0]_{\mathfrak{s}_{\mathfrak{r}_i}}$.
After reaching consensus in equations (\ref{eq: MPC ADMM consensus}), one can use $[\mathbf{\Phi}]_{i_{r}}^{k+1}$ in algorithm (\ref{eq: MPC ADMM localized 1}). Therefore by solving iterate sub-problem (\ref{eq: MPC ADMM localized 1 - row}) using the ADMM based consensus-like updates (\ref{eq: MPC ADMM consensus}), we are able to accommodate $d$-local coupling introduced in the constraints and objective function while still only exchanging information with $d$-local neighbors.  The rest of the analysis follows just as in the previous subsection.

Algorithm \ref{alg: II} summarizes the general DLMPC algorithm as implemented at subsystem $i$. 

\begin{algorithm}[h]
\caption{Subsystem $i$ implementation of DLMPC general subject to localized coupling}\label{alg: II}
\begin{algorithmic}[1]
\State \textbf{input:} convergence tolerance parameters, $\epsilon_p$, $\epsilon_d$, $\epsilon_x >0$.
\State Measure local state $[x_0]_{i}$.
\State Share measurement with $\textbf{out}_{i}(d)$.
\State Solve optimization  problem (\ref{eq: MPC ADMM consensus 1}).
\State Share $[\mathbf{X}]_{i}^{n+1}$ with $\textbf{out}_{i}(d)$. Receive the corresponding $[\mathbf{X}]_{j}^{n+1}$ from $\textbf{in}_{i}(d)$.
\State Perform update (\ref{eq: MPC ADMM consensus 2}).
\State Share $[\mathbf{Z}]_{i}^{n+1}$ with $\textbf{out}_{i}(d)$. Receive the corresponding $[\mathbf{Z}]_{j}^{n+1}$ from $\textbf{in}_{i}(d)$.
\State Perform update (\ref{eq: MPC ADMM consensus 3}).
\State If $\left\|[\mathbf{X}]^{n+1}_i - [\mathbf{Z}]^{n+1}_i\right\|_F<\epsilon_x$ go to step 10, otherwise return to step 4.
\State Share $[\mathbf{\Phi}]_{i_{r}}^{k+1}$ with $\textbf{out}_{i}(d)$. Receive the corresponding $[\mathbf{\Phi}]_{j_{r}}^{k+1}$ from $\textbf{in}_{i}(d)$and build $[\mathbf{\Phi}]_{i_{c}}^{k+1}$.
\State Solve optimization problem (\ref{eq: MPC ADMM localized 1 - column}) via the closed form solution (\ref{eq: MPC ADMM localized 1 - proxi}).
\State Share $[\mathbf{\Psi}]_{i_{c}}^{k+1}$ with $\textbf{out}_{i}(d)$. Receive the corresponding $[\mathbf{\Psi}]_{j_{c}}^{k+1}$ from $\textbf{in}_{i}(d)$ and build $[\mathbf{\Psi}]_{i_{r}}^{k+1}$.
\State Perform the multiplier update (\ref{eq: MPC ADMM localized 1 - lagrange}).
\State Check convergence as $\left\Vert[\mathbf{\Phi}]_{i_{r}}^{k+1}-[\mathbf{\Psi}]_{i_{r}}^{k+1}\right\Vert_F\leq\epsilon_{p}$ and $\left\Vert[\mathbf{\Psi}]_{i_{r}}^{k+1}-[\mathbf{\Psi}]_{i_{r}}^{k}\right\Vert_F\leq\epsilon_{d}$.
\State If converged, apply computed control action $[u_0]_i = [\Phi_u^{0,0}]_{i_{r}}[x_0]_{\mathfrak s_{\mathfrak{r_i}}}$, 
and return to step 2, otherwise return to step 4.
\end{algorithmic}
\end{algorithm}

\begin{remark}
Notice that in problem (\ref{eq: MPC ADMM consensus}) each subsystem $i$ optimizes for $[\mathbf{X}]_{i_{s}}$, while in Algorithm \ref{alg: II} subsystems only exchange $[\mathbf{X}]_{i}$. This is standard in consensus algorithms where only the components of $\mathbf{X}$ corresponding to subsystem $i$, $[\mathbf{X}]_{i}$ needs to be exchanged. 
\end{remark}

\subsubsection*{Computational complexity and convergence guarantees}
The presence of the coupling inevitably results in an increase of the computation and communication complexity of the algorithm. The computational complexity of the algorithm is now determined by steps 4, 6, 8, 11 and 13. All of these except for step 4 can be solved in closed form. Once again, by Assumption \ref{assump: locality} all sub-problems are over $O(d^{2}T)$ optimization variables and $O(dT)$ constraints, so the complexity does not increase with the size of the network. There is however an increased computational burden due to the nested consensus-like algorithm used to solve iterate sub-problem \eqref{eq: MPC ADMM global 2 - row}, which leads to an increase in the number of iterations needed for convergence. This also results in increased communication between subsystems, as local information exchange is needed as part of the consensus-like step as well. However, once again this exchange is limited to within a $d$-local subset of the system, resulting in small consensus problems that converge quickly, as we illustrate empirically in the next section.

Since Algorithm \ref{alg: II} is identical to Algorithm \ref{alg: I} save for the approach to solving the first iterate sub-problem, convergence follows from a similar argument as that used to prove Corollary 2.1. The same argument for recursive feasibility and stability expressed in the previous subsection holds for Algorithm \ref{alg: II}.

\section{SIMULATION EXPERIMENTS} \label{simulation}

In this section we illustrate the benefits of DLMPC as applied to large-scale distributed systems. To do so, we consider two case-studies. In the first, we aim to illustrate how the proposed algorithm accounts for general convex coupling objective functions and constraints. In the second, we empirically characterize the computational complexity properties of DLMPC.  All code needed to replicate these experiments is available at \url{https://github.com/unstable-zeros/dl-mpc-sls}.

\subsection{Optimization features}

In this example we use a dynamical system consisting of a chain of four pendulums coupled through a spring ($1N/m$) and a damper ($3Ns/m$). Each pendulum is modeled as a two-state subsystem described by its angle $\theta$ and angular velocity $\dot{\theta}$. Each of the pendulums can actuate its velocity. The simulations are done with a prediction horizon of $T=10s$, and we impose a localized region of size $d=1$ subsystems. The initial condition is arbitrarily generated with MATLAB $\mathrm{rng}(2020)$.

In order to illustrate how the presented algorithm can accommodate local coupling between subsystems as introduced by the objective function and constraints, we begin with the following example. We compare the control performance achieved by solutions to the DLMPC sub-problem \eqref{eq: MPC without ADMM} computed as a centralized problem using the Gurobi solver \cite{noauthor_gurobi_nodate} and CVX interpreter \cite{grant_cvx:_2013}, \cite{grant_graph_2008} (dotted line), and using Algorithm \ref{alg: I} or \ref{alg: II} (solid line), as appropriate, in Figure \ref{fig: dynamics}.  In particular, we plot the evolution of the position of the first two pendulums under different control objectives and constraints. In scenario 1 we consider the quadratic cost $f(\mathbf x,\mathbf u) = \sum_{i=1}^4 \|[\mathbf x]_i\|_2^2+\|[\mathbf u]_i\|_2^2$, and have no additional constraints. In Scenario 2 we consider a quadratic cost coupling the angle of adjacent pendulums, i.e. the control objective is a sum of functionals of the form $f([\theta]_{i},[u]_{i})=([\theta]_{i}-\frac{1}{2}\sum[\theta]_{j})^{2}+[\dot{\theta}]_{i}^{2}+[u]_{i}^{2}$. Finally, Scenario 3 uses the same objective function as Scenario 2 and further constrains the maximal allowable deviation between subsystem angles, i.e., $\vert[\theta]_{i}-[\theta]_{j}\vert\leq0.05$, for all $t>2$. Once again, the centralized solution coincides with the solution achieved by either Algorithm \ref{alg: I} or \ref{alg: II}, again validating the optimality of the algorithms proposed. 

\begin{figure}[htp]
    \centering
     \includegraphics[width=12cm, page=4]{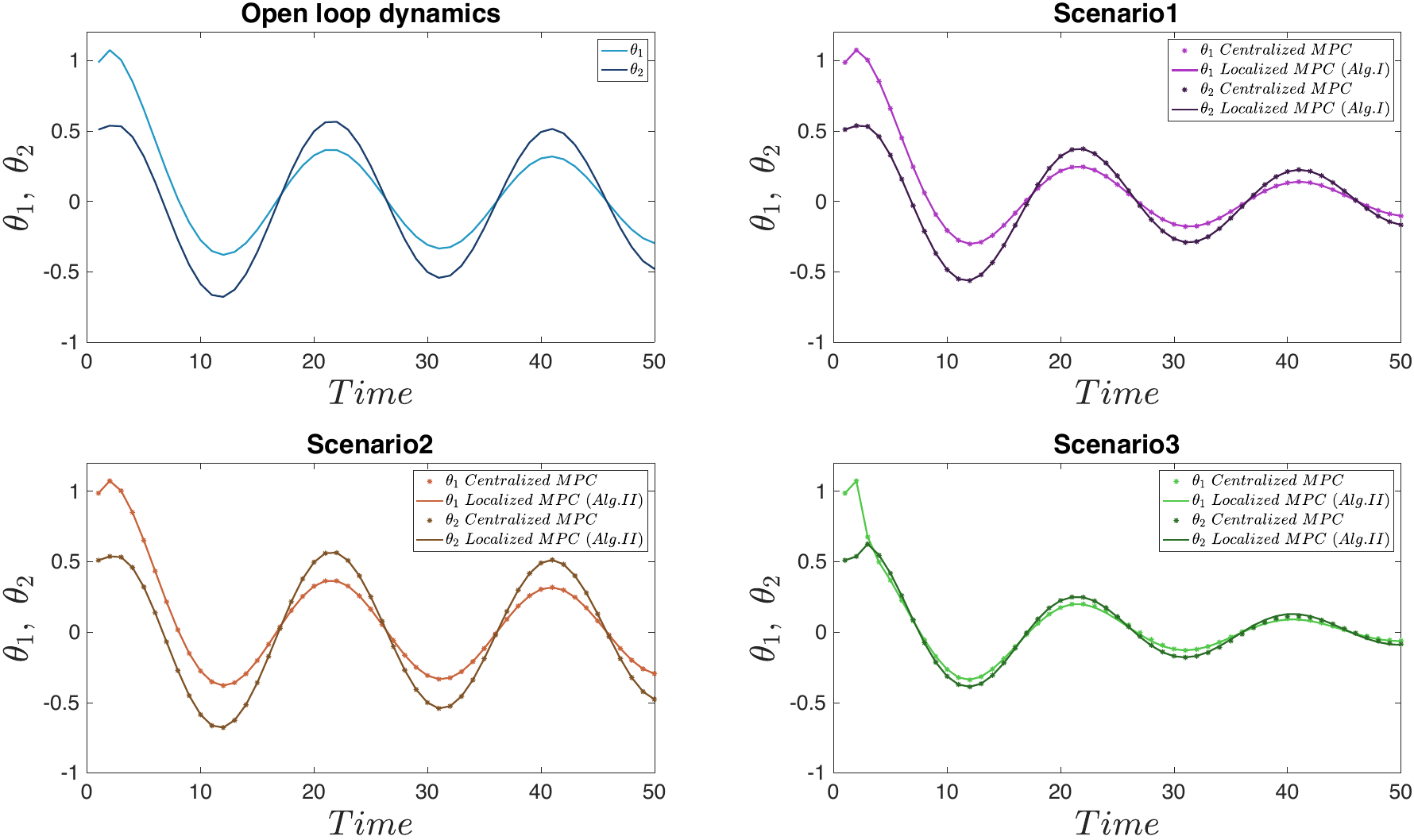}
    \caption{
    The evolution of the position of the first two pendulums in open loop is shown in the top left figure, whereas the top right most figure shows the position of the first two pendulums under MPC control with a quadratic penalty on state and input, and no constraints (scenario 1). The bottom left figure shows the position of the first two pendulums when the performance objective couples the angle of adjacent pendulums (scenario 2). On the bottom right, the position of the first two pendulums when the performance objective and the constraints couple adjacent pendulums (scenario 3).}
    \label{fig: dynamics}
\end{figure}

In order to further illustrate the impact of introducing a terminal constraint, we use Scenario 1 in Figure \ref{fig: dynamics} and we simulate the closed loop dynamics with the DLMPC controller both with and without terminal constraint $x_T=0$. As can be seen in Figure \ref{fig: dynamics1}, the system with terminal constraint reaches steady state faster at the expense of a higher optimal cost, specially at short time horizons. In the two closed loop scenarios presented in Figure \ref{fig: dynamics1}, we show both the control performance achieved by the solutions to the DLMPC subproblem \eqref{eq: MPC without ADMM} using Algorithm \ref{alg: I} and the solution to the standard MPC subproblem \eqref{eq: MPC} computed as a centralized problem using the Gurobi solver \cite{noauthor_gurobi_nodate} and CVX interpreter \cite{grant_cvx:_2013}. Again, in both cases the centralized solution coincides with the solution achieved by Algorithm \ref{alg: I}, so the solution achieved by Algorithm \ref{alg: I} is optimal.

\begin{figure}[htp]
\centering
\includegraphics[width=12cm, page=5]{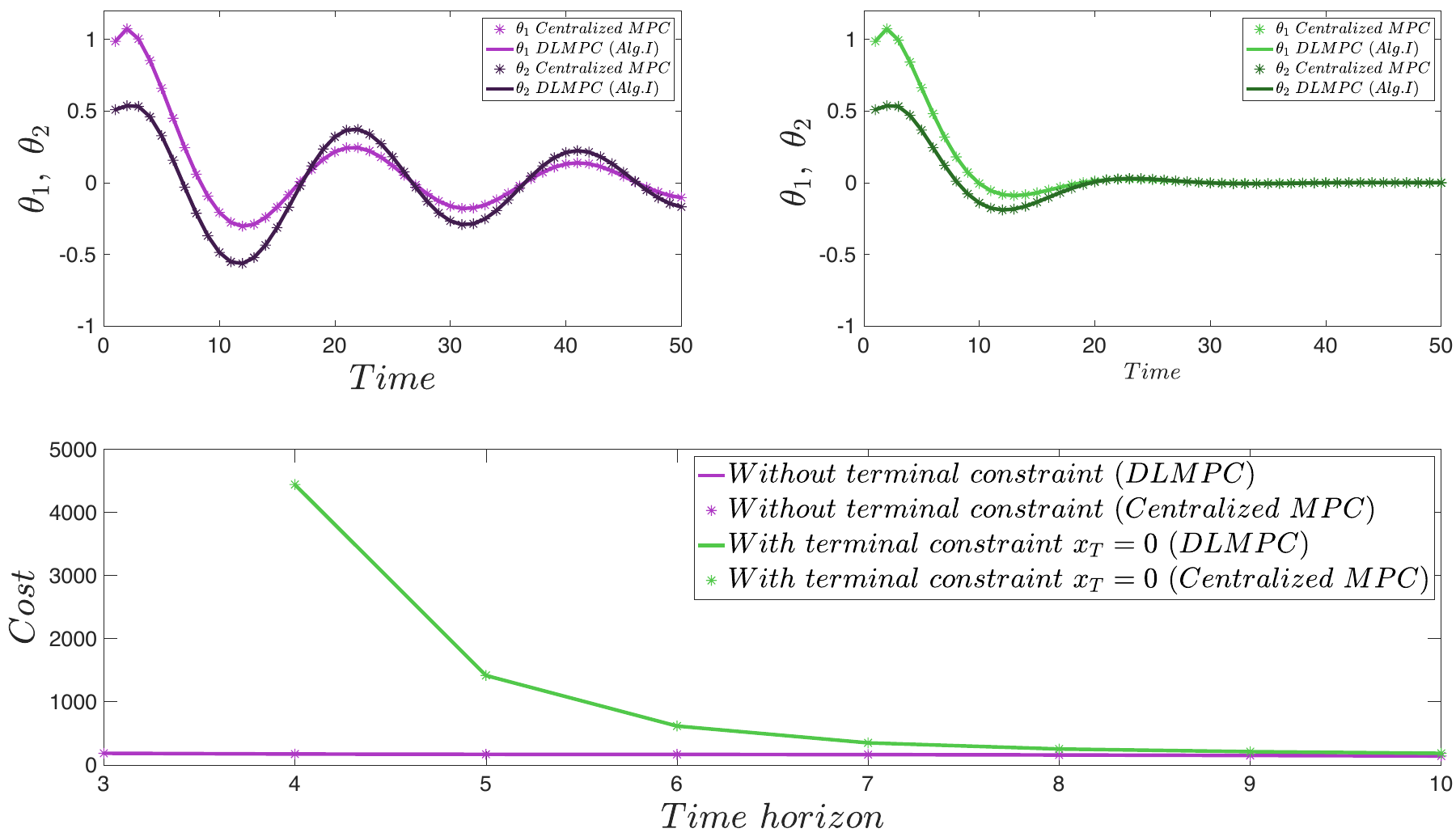}
\caption{
On the top from left to right: evolution of the position of the first two pendulums in open loop, with MPC control and no terminal constraint, and with MPC control with terminal constraint. Both cases under the quadratic cost $f(\mathbf x,\mathbf u) = \sum_{i=1}^4 \|[\mathbf x]_i\|_2^2+\|[\mathbf u]_i\|_2^2$ and $T=10$. On the bottom, the cost of both MPC setups with and without terminal constraints.}
\label{fig: dynamics1}
\end{figure}

\subsection{Per subsystem computational complexity}

Here we show how Algorithms \ref{alg: I} and \ref{alg: II} allow DLMPC to be applied to large-scale systems. Here, we let the subsystem dynamics be described by \[
[x(t+1)]_{i}=[A]_{ii}[x(t)]_{i}+\sum_{j\in\textbf{in}_{i}(d)}[A]_{ij}[x(t)]_{j}+[B]_{ii}[u(t)]_{i}\],
where 
\[[A]_{ii}=\begin{bmatrix}
   1  & 0.1 \\
  -0.3 &  0.7
\end{bmatrix}, \ [A]_{ij}=\begin{bmatrix}
   0  & 0 \\
   0.1 &  0.1
\end{bmatrix}, \ [B]_{ii}=\begin{bmatrix}
   0 \\
   0.1 
\end{bmatrix}.
\] 

The MPC horizon is $T=5$. We present four different scenarios that encompass different degrees of computational complexity of Algorithms \ref{alg: I} and \ref{alg: II}:

\begin{itemize}
    \item Case 1: per subsystem separable quadratic cost and no constraints. 
    \item Case 2:  per subsystem separable quadratic cost and per subsystem separable constraints.
    \item Case 3: quadratic cost coupling $d$-local subsystems and no constraints.
    \item Case 4: quadratic cost and polytopic constraints coupling $d$-local subsystems. 
\end{itemize}{}

The computational complexity of each of the cases is determined by (i) if the row-wise iteration sub-problem can be solved in closed form, and (ii) if Algorithm \ref{alg: I} or \ref{alg: II} is needed -- we summarize these properties for the four cases described above in Table \ref{tab: complexity}.

\begin{table}[h]
\centering
\begin{tabular}{ |c|c|c| } 
 \hline
 \hline
  Case & Algorithm & Computation of step 4 \\
 \hline
 \hline
  1 & \ref{alg: I} & Closed form \\
 \hline
  2 & \ref{alg: I} & Needs minimization solver\\
 \hline
  3  & \ref{alg: II} & Closed form \\
 \hline
  4  & \ref{alg: II} & Needs minimization solver\\
 \hline
\end{tabular}
\caption{Summary of cases considered in Section 5.2.}
\label{tab: complexity}
\end{table}

As in \cite{conte_computational_2012}, we characterize the runtime per MPC iteration of the DLMPC algorithms\footnote{Runtime is measured after the first iteration, so that all the iterations for which runtime is measured are warmstarted}. We measure runtime per state as opposed to per subsystem since the Algorithm presented performs iterations row-wise and column-wise, each corresponding to a state. In the leftmost plot of Figure \ref{fig: scalability}, we fix the locality parameter as $d=1$, and demonstrate that the runtime of both Algorithms \ref{alg: I} and \ref{alg: II} does not increase with the size of the network, assuming each subsystem is solving their sub-problems in parallel. These observations are consistent with those of \cite{conte_computational_2012} where the same trend was noted. The slight increase in runtime - in particular for Cases 3 and 4 - we conjecture is due to the introduced coupling, as the more subsystems are coupled together the longer it takes for the consensus-like sub-routine to converge. In any case, the increase in runtime does not seem to be significant and appears to level off for sufficiently large networks.

\begin{figure}[htp]
    \centering
    \includegraphics[width=12cm, page=6]{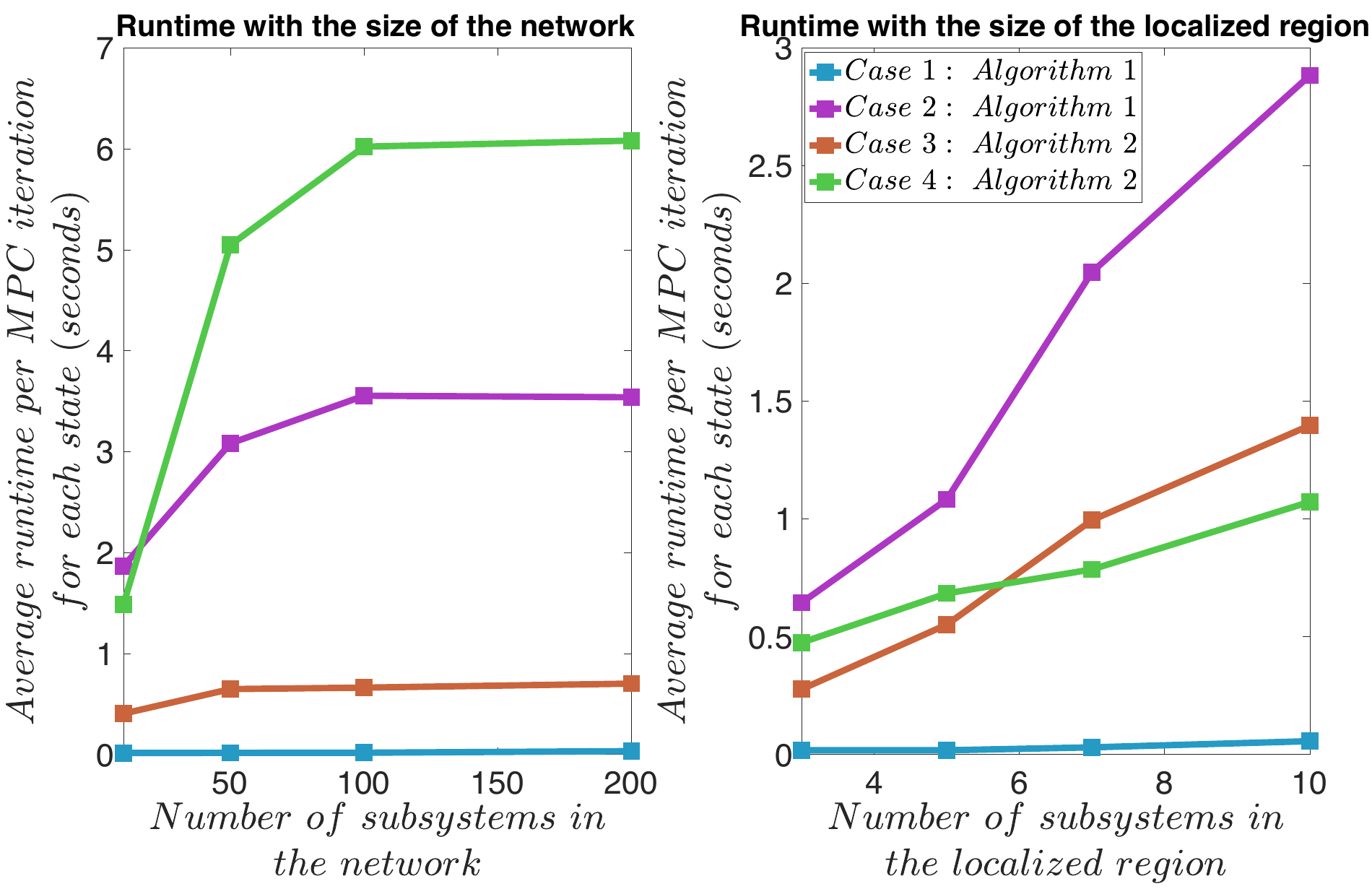}
   \caption{On the left, the runtime of each of the four different cases for different network sizes. On the right, the runtime of each of the four different cases for different sizes of the localized region.}
    \label{fig: scalability}
\end{figure}

In the rightmost plot of Figure \ref{fig: scalability}, we fix the number of systems to $N=10$, and explore the effect of the size of the localized region $d$ on computational complexity. While a larger localized region $d$ can lead to improved performance, as a broader set of subsystems can coordinate their actions directly, it also leads to an increase in computational complexity, as the number of optimization variables per sub-problem scales as $O(d^{2}T)$. Moreover, a larger localized region results in larger consensus-like problems being solved as a sub-routine in Algorithm \ref{alg: II}, further contributing to a larger runtime.  Thus by choosing the smallest localization parameter $d$ such that acceptable performance is achieved, the designer can tradeoff between computational complexity and closed loop performance in a principled way.  For the particular system studied in these simulations, performance did not change with the parameter $d$ and future work will look to understand this better. We conjecture that it will play an important role in the performance-complexity tradeoff when disturbances are present. This further highlights the importance of exploiting the underlying structure of the dynamics, which allow us to enforce locality constraints on the system responses, and consequently, on the controller implementation. \vspace*{3mm}

\section{CONCLUSIONS}

We defined and analyzed a \emph{closed loop} Distributed and Localized MPC algorithm. By leveraging the SLS framework, we were able to enforce information exchange constraints by imposing locality constraints on the system responses.  We further showed that when locality is combined with mild assumptions on the separability structure of the objective functions and constraints the problem, an ADMM based solution to the DLMPC subproblems can be implemented that requires only local information exchange and system models, making the approach suitable for large-scale distributed systems. Moreover, our approach can accommodate constraints and objective functions that introduce local coupling between subsystems through the use of a consensus-like algorithm. To be best of our knowledge, this is the first DMPC algorithm that allows for the distributed synthesis of closed loop policies. 

In future work, we plan to develop robust variants of DLMPC that can accommodate additive perturbations,  model uncertainty, and approximately localizable systems, by leveraging the robust variants of the SLS parameterization \cite{matni_scalable_2017}.  We will also explore whether locality constraints allow for a scalable computation of robust invariant sets for large-scale distributed systems, as well as their implications on the complexity of (approximate) explicit MPC approaches.  

Finally, it is of interest to extend the results presented in this paper to information exchange topologies defined in terms of both sparsity and delays -- while the SLS framework naturally allows for delay to be imposed on the implementation structure of a distributed controller, it is less clear how to incorporate such constraints in a distributed optimization scheme in a distributed optimization scheme.





\subsubsection*{Acknowledgement}
 The authors thank Manfred Morari for invaluable comments and helpful discussions.
\bibliographystyle{IEEEtran}
\bibliography{references}

\end{document}